\documentclass[reqno,11pt]{amsart}
\usepackage{amsmath, latexsym, amsfonts, mathabx, amsthm, amscd}
\usepackage{multirow,stmaryrd}
\usepackage{enumitem}
\usepackage{color}
\usepackage{graphics,epsf,psfrag}
\usepackage{graphicx}
\numberwithin{equation}{section}

\newtheorem{theorem}{Theorem}[section]
\newtheorem{lemma}[theorem]{Lemma}
\newtheorem{proposition}[theorem]{Proposition}

\newtheorem{rem}[theorem]{Remark}
\newtheorem{definition}[theorem]{Definition}

\newcommand{\ind}{\mathbf{1}}

\renewcommand{\tilde}{\widetilde}
\renewcommand{\hat}{\widehat}

\newcommand{\cE}{{\ensuremath{\mathcal E}} }

\newcommand{\cN}{{\ensuremath{\mathcal N}} }

\newcommand{\cM}{{\ensuremath{\mathcal M}} }

\newcommand{\bP}{{\ensuremath{\mathbf P}} }
\newcommand{\bE}{{\ensuremath{\mathbf E}} }

\DeclareMathSymbol{\leqslant}{\mathalpha}{AMSa}{"36} % nicer `smaller or equal'
\DeclareMathSymbol{\geqslant}{\mathalpha}{AMSa}{"3E} % nicer `larger or equal'
\DeclareMathSymbol{\eset}{\mathalpha}{AMSb}{"3F}     % nicer `emptyset'
\renewcommand{\leq}{\;\leqslant\;}                   % redef. of < or =
\renewcommand{\geq}{\;\geqslant\;}                   % redef. of > or =
\newcommand{\dd}{\,\text{\rm d}}             % a straight d for differentials
\DeclareMathOperator*{\union}{\bigcup}       % \sum-like symbol for union
\newcommand{\sumtwo}[2]{\sum_{\substack{#1 \\ #2}}} % sum with 2 lines
\newcommand{\prodtwo}[2]{\prod_{\substack{#1 \\ #2}}}     % product 2 lines
\newcommand{\bbE}{{\ensuremath{\mathbb E}} }

\newcommand{\bbN}{{\ensuremath{\mathbb N}} }

\newcommand{\bbP}{{\ensuremath{\mathbb P}} }

\newcommand{\bbR}{{\ensuremath{\mathbb R}} }

\newcommand{\bbZ}{{\ensuremath{\mathbb Z}} }

\newcommand{\ga}{\alpha}
\newcommand{\gb}{\beta}
\newcommand{\gd}{\delta}
\newcommand{\gep}{\varepsilon}       % \ge already exists...
\newcommand{\gp}{\varphi}

\newcommand{\gD}{\Delta}

\newcommand{\go}{\omega}

\newcommand{\gO}{\Omega}
\newcommand{\gl}{\lambda}

\newcommand{\gS}{\Sigma}

\makeatletter
\def\captionfont@{\footnotesize}
\def\captionheadfont@{\scshape}

\long\def\@makecaption#1#2{%
  \vspace{2mm}
  \setbox\@tempboxa\vbox{\color@setgroup
    \advance\hsize-6pc\noindent
    \captionfont@\captionheadfont@#1\@xp\@ifnotempty\@xp
        {\@cdr#2\@nil}{.\captionfont@\upshape\enspace#2}%
    \unskip\kern-6pc\par
    \global\setbox\@ne\lastbox\color@endgroup}%
  \ifhbox\@ne % the normal case
    \setbox\@ne\hbox{\unhbox\@ne\unskip\unskip\unpenalty\unkern}%
  \fi
  \ifdim\wd\@tempboxa=\z@ % this means caption will fit on one line
    \setbox\@ne\hbox to\columnwidth{\hss\kern-6pc\box\@ne\hss}%
  \else % tempboxa contained more than one line
    \setbox\@ne\vbox{\unvbox\@tempboxa\parskip\z@skip
        \noindent\unhbox\@ne\advance\hsize-6pc\par}%
\fi
  \ifnum\@tempcnta<64 % if the float IS a figure...
    \addvspace\abovecaptionskip
    \moveright 3pc\box\@ne
  \else % if the float IS NOT a figure...
    \moveright 3pc\box\@ne
    \nobreak
    \vskip\belowcaptionskip
  \fi
\relax
}
\makeatother
\def\writefig#1 #2 #3 {\rlap{\kern #1 truecm
\raise #2 truecm \hbox{#3}}}

\newcommand{\tf}{\textsc{f}}
\newcommand{\tg}{\textsc{g}}

\newcommand{\lint}{\llbracket}
\newcommand{\rint}{\rrbracket}
\newcommand{\cc}{\complement}
\title[The critical behavior of the  {\small $\ga=0$}  copolymer model]{
Disorder and critical phenomena: \\
the  {\large $\alpha=0$} copolymer model 
}

\author{Quentin Berger}
\address{
  Universit\'e Pierre et Marie Curie, Laboratoire de Probabilit{\'e}s et Mod\`eles Al\'eatoires, UMR 7599,
            F- 75205 Paris,France
}
\email{quentin.berger@upmc.fr}

\author{Giambattista Giacomin}
\address{
  Universit\'e Paris Diderot, Sorbonne Paris Cit\'e,  Laboratoire de Probabilit{\'e}s et Mod\`eles Al\'eatoires, UMR 7599,
            F- 75205 Paris, France
}
\email{giambattista.giacomin@univ-paris-diderot.fr}

\author{Hubert Lacoin}
\address{  IMPA-Instituto de Matem\'atica Pura e Aplicada, Estrada Dona Castorina 110
Rio de Janeiro, CEP-22460-320, Brasil. 
            }
\email{lacoin@impa.br}

\begin{document}

\begin{abstract}
The generalized copolymer model is a disordered system built on a discrete renewal process with inter-arrival distribution that decays in a regularly varying fashion with exponent $1+ \alpha\geq 1$. It exhibits a localization transition which can be characterized in terms of the free energy of the model: the free energy is zero in the delocalized phase and it is positive in the localized phase. This transition, which is observed when tuning the mean $h$ of the disorder variable, has been tackled in the physics literature notably via a renormalization group procedure  that goes under the name of  \emph{strong disorder renormalization}. We focus on the case $\alpha=0$ --  the critical value $h_c(\beta)$ of the parameter $h$ 
%for the copolymer model 
is exactly known (for every strength $\beta$ of the disorder) in this case -- and we provide precise estimates on the critical behavior. Our results confirm the strong disorder renormalization group prediction that the transition is of infinite order, namely that when $h\searrow h_c(\beta)$
the free energy vanishes 
faster than any power of $h-h_c(\beta)$. But we   show %, in full generality, 
 that the free energy vanishes 
much faster than the physicists' prediction.

\medskip
\noindent  \emph{AMS  subject classification (2010 MSC)}:
60K37,  %Processes in random environments
82B27, % Critical phenomena
82B44, % Disordered systems
60K35, % Stat mech. type models
82D60. % Polymers

\end{abstract}

\maketitle
%\tableofcontents

\section{Introduction}
\label{sec:intro}

The effect of disorder on phase transitions and critical phenomena is a key issue to which a considerable attention has been paid in the physical literature and, more recently, also in the mathematical one. We refer to \cite{cf:Bovier,cf:G} for a (necessarily partial)   overview of this vast  subject
from mathematicians' perspective.
One of the basic questions -- \emph{do phase transitions withstand the introduction of impurities, i.e.\ disorder, and, if it is so, is the critical behavior the same as the one of the  pure model or not?} -- has been repeatedly at the center of the attention. On the other hand there are also cases 
in which the disorder induces a phase transition that is absent in the pure model (this is for example the case 
for the Directed Polymers in Random Environment in dimension three or larger, see e.g.~\cite{cf:C}). 

In spite of a number of remarkable results -- let us cite for instance the works \cite{cf:AW,cf:BK} on  
 the  Ising model  with  random external field -- 
the issue of understanding 
 the critical behavior in presence of disorder  is very delicate and certainly little understood  
 (at times even the issue of whether there is any transition at all is out of reach or far from obvious: \cite{cf:AW,cf:BK} are good examples of this).
The 
 renormalization group (RG) approach proposed  by 
A.~B.~Harris \cite{cf:Harris}
 turned out to be quite successful  from a physics standpoint  for a considerable number of disordered systems. 
 It is helpful for us to take the Harris' viewpoint at least for exposition purposes. Harris' approach 
 demands a model in which: (i) the disorder can be made small and switched off by tuning a parameter; (ii) the non-disordered (or \emph{pure}) model displays a phase transition.  Then  we can consider tackling the issue in a perturbative way and 
ask what is the effect of a small amount of disorder. Following \cite{cf:Harris}, it is   customary to say that the disorder is \emph{irrelevant} if the 
action of the RG makes the disorder weaker and we say that it is \emph{relevant} if the disorder is enhanced by the transformation. Making a long story short, in the first case the   
phase transition persists and disorder essentially does not alter the  critical behavior  (for example: unchanged critical  exponent(s)), while in the second case 
it is reasonable to expect a different critical behavior and possibly even that the transition is washed out.  Of course this scenario  is not the most general and is far from being mathematically understood. But the notion of relevant disorder  clearly identifies   cases in which the nature of the critical behavior  is determined by the disorder. 
In this sense, possibly going beyond the scope of \cite{cf:Harris}, it is natural to consider that disorder is relevant also in cases 
like the one in \cite{cf:AW} where the transition is smoothed out, possibly to the level of washing it out completely (even if this issue is open).

Our work is about a relevant disorder case and aims at determining in a precise way  the nature of the critical phenomenon.
In this direction, results  have been obtained recently in the context of localization transitions for random polymers and interfaces, notably \emph{copolymer near selective interface} models (copolymer for short) and \emph{pinning} models.
These models depend on two parameters: $\gb\ge 0$ that controls the strength of the disorder (if $\gb=0$, the model is pure) and $h\in \bbR$ which plays in favor, 
respectively against, localization if it is larger, respectively smaller than zero. In the limit of infinitely large systems, there is a critical value $h_c(\gb)$ 
such that the model is in a localized state if $h>h_c(\gb)$ and it is in a delocalized state if $h < h_c(\gb)$: in these models the transition can be characterized in a very simple way in terms of the free energy, namely the free energy is non negative and
localization is equivalent to the positivity of the free energy. 
 This class of models is particularly interesting because of the numerous applications, but also because  several (not always compatible) predictions have been set forth by physicists. But it seems now that a certain agreement has grown about the fact that these models fall into the realm %, or universality class, 
 of the
\emph{(real space) strong disorder  RG } approach, also known as Ma-Dasgupta RG, from S. K. Ma and C. Dasgupta
who first proposed to abandon the  \emph{regular} procedures (block averaging, sub-lattice decimation,...), often performed by taking advantage of 
transforms and working for example in Fourier space, and focus on the irregular nature of the disorder. This is implemented by performing a coarse graining starting  with sites on which the disorder is larger.
These ideas were substantially developed by D. Fisher \cite{cf:dfisher} who realized that  this procedure can yield exact results, and who set forth the idea that, iterating the strong disorder  RG, one may end up on an \emph{infinite disorder}  fixed point or 
on a \emph{finite disorder} fixed point: exact results are expected in the first case. We retain of this approach 
that  it predicts for copolymer and pinning models that the transition becomes of \emph{infinite order} in the sense that the free energy for $h \searrow h_c(\gb)$
vanishes faster than any power of $(h-h_c(\gb))$ \cite{cf:IM, cf:M1, cf:M2}. A result of this type has been established in \cite{cf:FF2}, for a pinning model based on the two dimensional free field,
hence a $2+1$-dimensional model: it is not clear  whether or not 
 this model can be understood via strong disorder  RG approach (this   approach has been almost always applied in cases in which  the disorder is one dimensional), but we point out that the disorder makes the transition of infinite order. For the pinning model based on higher dimensional free fields \cite{cf:GL} -- d+1 dimensions with $d \ge 3$ -- the result is different, with a milder smoothing phenomenon. The articles \cite{cf:GL,cf:FF2}
are up to now  the only cases in which the critical exponents in presence of relevant disorder have been determined .  

Here we present a third case: a special case of the copolymer model. 
%This model falls into the class of the strong disorder RG and  sharply addresses some predictions in physics.
 %, the only one for which the critical point is explicitly known. 
The copolymer model  has been tackled via  the strong disorder  RG approach with precise claims  \cite{cf:IM, cf:M1}. Our results are in agreement with the fact   that the  transition is of infinite order, but as we will explain, from a finer perspective this agreement is lost.
 %the physicists' prediction differs from ours. 

\subsection{The $\ga=0$ copolymer model and the localization transition}

We work with IID disorder $\{\go_n\}_{n=1,2, \ldots}$ of law $\bbP$ and we use  the notation
\begin{equation}\label{defgl}
\gl(\gb):= \log \bbE \exp(\gb \go_1).
\end{equation}
We assume that $\overline{\gb}:= \sup \{\gb\in \bbR:\, \gl(\gb)< \infty\}\in (0 , \infty]$, and for the sake of normalization
\begin{equation}
\label{normalize}
\bbE[\go_1^2]\,=\,1
 \ \ \text{ and } \ \    \bbE[\go_1]\,=\, 0 \,.
\end{equation}
We consider the discrete probability density $K(n)=L(n)/n$, $n\in \bbN:=\{1,2, \ldots\}$, with $L(\cdot)$ slowly varying at $+\infty$  and such that $\sum_{n=1}^\infty K(n)=1$. $L: (0, \infty) \to (0, \infty)$ is slowly varying \cite{cf:RegVar} if it is measurable and if $\lim_{x \to \infty} L(cx)/L(x)=1$ for every $c>0$:  
examples are presented in Definition~\ref{def:L} and, without loss of generality for us,  we can assume $L(\cdot)$ to be  smooth
\cite[Th.~1.3.3]{cf:RegVar}.
$\bP$ is the law of the two  random sequences $\tau = \{\tau_j\}_{j=0,1, \ldots}$ and $\iota=\{\iota_{j=1,2,  \ldots}\}$ with $\tau$ and $\iota$ independent and 
\begin{itemize}
\item  
$\tau$ is a renewal sequence  with $\tau_0=0$ and inter-arrival distribution $K(\cdot)$; 
\item 
$\iota$ is a sequence of independent Bernoulli variables of parameter $1/2$.
\end{itemize}
\smallskip 
Given $(\tau, \iota)$ we say, for $j=1, 2, \ldots$,  that $\{\tau_{j-1}+1, \ldots, \tau_{j}\}$ is the $j^{\textrm{th}}$ excursion and this excursion
is \emph{below the interface} (respectively, \emph{above the interface})  if $\iota_j=1$ (respectively, $\iota_j=0$).
Given $n=0,1, \ldots$ there exists a unique $j=j(n)$ such that $n \in \{\tau_{j-1}+1, \ldots, \tau_{j}\}$:
we then define $\gD_n=\iota_{j(n)}$, so we have also that $\gD_{\tau_j+1}=\gD_{\tau_{j+1}}= \iota_{j+1}$. Note that, once $\gD=\{\gD_n\}_{n=1,2, \ldots}$ is introduced,  the process $(\tau, \iota)$ is equivalent 
to the process $(\tau, \gD)$, and we will mostly prefer to use this second representation.

The partition function and central expression for our analysis is for $N\in \bbN$
\begin{equation}
\label{eq:Zngo}
Z_{N, \go}\,:=\, \bE \bigg[
\exp \bigg( \sum_{n=1}^N \left( \gb\go_n - \gl(\gb)+h\right) \gD_n \bigg) ;\, 
N \in \tau \bigg]\, ,
\end{equation}
and of course this defines a statistical mechanics model. We set also $Z_{0, \go}:=1$.
A way of thinking of it is to consider  a directed $1+1$ dimensional polymer which touches (and possibly crosses,
but not necessarily) a flat  interface -- the horizontal axis -- that separates the two half planes that are full
of two different solvents. Each monomer carries a charge  $\gb\go_n - \gl(\gb)+h$ which can be positive or negative and, while the solvent above the interface does not interact with the monomers, the solvent below the interface 
favors the positively
charged monomers and penalizes the negatively charged ones. This can be read out of \eqref{eq:Zngo} once we stipulate
that excursions with $\gD=0$ , respectively $\gD=1$, are above, respectively  below, the interface. 

%Note that modifying the law of the renewal process $\tau$ via such an exponential term (or Boltzmann factor) may lead to a localization phenomenon. To understand what we mean note that the random set $\tau=\{\tau_0, \tau_1, \ldots\}$ contains infinitely many points, even if it has zero density. In fact, by the Renewal Theorem, almost surely $\lim_{N \to \infty} \vert \tau \cap [0, N] \vert /N =1/\bE[\tau]=0$.  But if we consider for example the case $\gb \go_n-\gl(\gb)+h$ that is mostly negative (this can be achieved by choosing $h$ negative and large: in the case of bounded variables $\go_n$ one can even make this quantity negative for every $n$) it may be advantageous for the trajectories to avoid the lower half plane and in this case the process $\tau$ may end up being essentially only one huge excursion in the upper half plane  (delocalized regime). On the other hand if $\gb \go_n-\gl(\gb)+h$ oscillates around zero it may be convenient  for excursion to adjust their length and whether they are above or below the interface to gain from the exponential weight: in this second case the renewal process would be modified to become a process with a positive density of contacts with the interface (localization). Of course for very large $h$ a similar argument in favor of a delocalized behavior (in the lower half plane) applies but we will focus only on on values of $h $ close to $\gl(\gb)$.

Consider now the 
free energy (density) 
\begin{equation}
\tf(\gb, h)\, :=\, \lim_{N \to \infty} \frac 1N \bbE \log Z_{N, \go}\, .
\end{equation}
The existence of the limit follows from the super-additivity of $\{ \bbE \log Z_{N, \go}\}_{N=1, 2, \ldots}$, see \cite[Chapter 4, \S 4.4]{cf:GB}. Moreover since $Z_{N, \go}\ge Z_{N, \go}( \tau_1=N, \gD_1=0)=K(N)/2$ 
one directly infers that $\tf(\gb, h)\ge 0$. On the other hand, for $\gb<\bar \gb$,
\begin{equation}
\label{eq:annealing}
\frac 1N\bbE \log Z_{N, \go} \, \le\, \frac 1N \log \bbE Z_{N, \go}\, =\, \frac 1N \log
\bE \Big[
\exp \Big( h\sum_{n=1}^N \gD_n \Big) ;\, 
N \in \tau \Big]
\stackrel{N \to \infty} \longrightarrow h \ind_{(0, \infty)}(h) \, ,
\end{equation}
where the upper bound corresponding to the limit is obvious and for the lower bound it suffices to restrict to
the events $\tau_1=N$ and $\iota_1=0$   (for $h\le 0$) or $1$  (for $h>0$). 
Therefore \eqref{eq:annealing}
implies $\tf(\gb, h)=0$ if $h\le 0$. It is easy to see that $\tf(\gb, h)>0$ for $h$ large, for instance by restricting the partition function to $\tau_1=N$ and $\iota_1=1$ (we get $\tf(\gb, h)\ge h-\lambda(\gb)$). 
This is already enough to claim that there exists a critical point $h_c(\gb)$, in the sense that
$\tf(\gb, \cdot)$ cannot be real analytic at $h_c(\gb)$, with $h_c(\gb):= \max\{h: \, \tf(\gb , h)=0\}$ (note that
the monotonicity of $\tf(\gb, \cdot)$ is obvious, as well as the convexity).
In fact $\tf(\gb, h)>0$ as soon as $h>0$ (this follows from the so called \emph{rare stretch strategy}, see \cite[Chapter~6]{cf:GB}, and it is
 also a byproduct of the proof in Section~\ref{sec:rs}) and therefore $h_c(\gb)=0$. 
As it is explained at length for example in \cite{cf:coprev}, the transition from  zero to positive free energy is, in a precise sense, a delocalization to localization transition, in particular in terms of path properties of the system. Here we just focus on the free energy and we note that the two rightmost terms in \eqref{eq:annealing} are the annealed free energy of the model for finite $N$ and for $N=\infty$. This is the pure model associated to the (quenched) disordered model 
we are analyzing, and it is therefore important to remark that the annealed model has a first order transition, i.e.\
the first derivative of the free energy is discontinuous. 

\subsection{The general copolymer and the pinning model}
For the sake of better understanding and motivating our results, it is important to consider a larger class of models.
First of all, the notation we use in this work  is not customary:  the copolymer model in the literature is typically introduced 
via the partition function 
\begin{equation}
\label{eq:Zngo-old0}
Z^{\textrm{cop}}_{N, \go}\,:=\, \bE \bigg[
\exp \bigg( \varrho \sum_{n=1}^N ( \go_n + h ) s_n \bigg) ;\, 
N \in \tau \bigg]\, ,
\end{equation}
where $s_n= 1-2 \gD_n \in \{-1, +1\}$, $\varrho\ge 0$ and $h\in \bbR$ (but $h \ge 0$ without loss of generality).
So the signs $s_n$ determine whether the excursions are above or below the interface.
%the energy entropy competition from which the transition arises is apparent. Namely, which configurations truly contribute to $Z^{\textrm{cop}}_{N, \go}$? Is there essentially  only one big excursion above the interface so that the energetic contribution is about $Nh$, with a mild  entropic loss because even in the extreme case of just une excursion above the interface the probability is $K(N)/2$ and the log of this quantity is about $ -\log N$? Or  the contributions  that tend to optimize te energetic gain by adapting length and sign of the excursions to the signs of the charges $\go_n+ h$ are going to give the main contribution? The set of all these trajectories is however exponentially small. 
The expression of the partition function in \eqref{eq:Zngo-old0} is the most natural for the interpretation of the model \cite{cf:GB},
but from the technical viewpoint it is very useful to observe that 
\begin{equation}
\label{eq:Zngo-old}
Z^{\textrm{cop}}_{N, \go} \exp \bigg( -\varrho  \sum_{n=1}^N \left( \go_n + h \right) \bigg) \,:=\, \bE \bigg[
\exp \bigg( -2 \varrho \sum_{n=1}^N \left( \go_n + h \right) \gD_n \bigg) ;\, 
N \in \tau \bigg]\, ,
\end{equation}
and the right-hand side is  a partition function that defines the same model because it differs from 
$Z^{\textrm{cop}}_{N, \go}$ only by a factor that does not depend on the process $(\tau, \iota)$.
Moreover it is evident that the right-hand side of \eqref{eq:Zngo-old}
 and $Z_{N, \go}$ are directly related by a change of variables.
The variables we use, i.e.\ $\gb $ and $h$, separate better the role of the parameter that mesures the strength of the disorder
($\varrho$ or $\gb$) and  the other parameter, $h$, with which one changes the average value of the charges.
Another reason to use \eqref{eq:Zngo-old} is for the formal similarity with 
\begin{equation}
\label{eq:Zngoprim}
Z^{\textrm{pin}}_{N, \go}\,:=\, \bE \bigg[
\exp \bigg( \sum_{n=1}^N ( \gb\go_n - \gl(\gb)+h) \gd_n \bigg) ;\, 
N \in \tau \bigg]\, ,
\end{equation}
where $\gd_n$ is the indicator function that $n\in \tau$, that is that there exists $j$ such that
$n=\tau_j$. $Z^{\textrm{pin}}_{N, \go}$ is the partition function of the pinning model  that displays a similar
localization transition, with the analogous critical point $h^{\textrm{pin}}_c(\gb)$. We refer to \cite{cf:G} for a complete introduction to the model and to the questions related to it.

For both  the copolymer and the pinning models it is natural to consider the general context of an inter-arrival distribution
$K(n)=L(n)/n^{1+ \ga}$, with $\ga\ge 0$. The pinning model exhibits a richer phenomenology than the copolymer model, in the sense 
that for the pinning model, disorder is irrelevant for $\ga < 1/2$, and it is relevant for $\ga>1/2$. 
A first illustration of this fact is that
when $\ga < 1/2$ we have $h^{\textrm{pin}}_c(\gb)=0$ (at least for $\gb$ small enough, and for all $\gb<\bar \gb$ in the case $\ga=0$, considered in \cite{cf:AZ}), while for $\ga >1/2$ we have $h^{\textrm{pin}}_c(\gb)<0$. But a more important point is that the critical behavior of the pure and disordered models coincide for $\ga <1/2$ and differ for $\ga >1/2$ (we refer to \cite{cf:BLmarg} for the state of the art in the marginal case $\ga=1/2$). We stress that neither the exact value of $h^{\textrm{pin}}_c(\gb)$ (see however \cite{cf:CTT}) nor the critical behavior is known for $\ga>1/2$, and the relevant character of the transition is established via a smoothing inequality \cite{cf:smooth} (see \cite{cf:smooth2} for a generalization) that implies that the critical behaviors of  disordered and pure systems differ. 

On the other hand, for the copolymer model, it is known that disorder is relevant for every $\ga \ge 0$: the free energy $\tf(0, h)$ of the pure model (which coincides with the annealed model)  is simply equal to
$h\ind_{(0, \infty)}(h)$ (first order phase transition), see 
\eqref{eq:annealing}, whereas the  free energy of the disordered model verifies $\tf (\gb, h)=O\left((h-h_c(\gb))^2\right)$ by the smoothing inequality \cite{cf:smooth,cf:smooth2}. Moreover, we have
$h_c(\gb)< 0$ -- bounds and even sharp bounds for $\gb \searrow 0$ are known \cite{cf:BCPSZ, cf:BoGi, cf:BGLT, cf:BdHO,cf:Bcoprev,cf:T_fractmom} but the exact value is unknown --
except (as we have already mentioned) for $\ga=0$ where $h_c(\gb)= 0$. 

We have reported only the mathematically rigorous results. About non rigorous approaches we refer to \cite[\S~6.4]{cf:G} for an
overview of  some claims made on criticality for copolymer models. Here we focus on the fact that both pinning and copolymer models are expected to fall into the universality classes of the strong disorder RG \cite{cf:dfisher,cf:IM,cf:Vojta}  that predicts that the transition becomes of infinite order if disorder is relevant. Precise claims in this direction are contained in \cite{cf:IM,cf:M1,cf:M2} where a 
critical behavior of the type $\exp(-c/(h-h_c(\gb))$, $c>0$, is predicted both for the copolymer 
\cite{cf:M1} and pinning models \cite{cf:M2}. For pinning models however, we find also the
prediction   $\exp(-c/\sqrt{h-h_c(\gb)})$ in \cite{cf:tangchate} and this latter claim reappears in \cite{cf:DR}, with arguments that  are still non rigorous but with a much richer  and  more convincing analysis developed for  a simplified version of the pinning model on hierarchical (diamond) lattices.

We present in the next sections results proving that for the $\ga=0$ copolymer the transition is of infinite order,
and in particular we show that the free energy close to criticality is 
  much smaller than
$\exp(-c/(h-h_c(\gb)))$, in the sense that $c$ should be replaced by a function that diverges as $h\searrow h_c(\gb)$.

\subsection{Main results}
We introduce 
$\tilde L(x):=\int_x^\infty (L(y)/y) \dd y$: by \cite[Prop.~1.5.9a]{cf:RegVar}  $\tilde L(\cdot)$ is slowly varying 
and  $\lim_{x \to \infty} \tilde L(x) /L(x)=\infty$.
The following definition identifies a useful framework of models:
 
\medskip

\begin{definition}
\label{def:L}
We say that the the decay at infinity of the slowly varying function function $L$:
\begin{itemize}
\item[(i)]  is \emph{sub-logarithmic} if $L$ satisfies 
\begin{equation}\label{slowdef}
L(x)=(1+o(1))\ c_L/(\log x ( \log \log x)^\upsilon) \, ,
\end{equation}
\item[(ii)]  is \emph{logarithmic} if  $L$ satisfies 
\begin{equation}\label{standarddef}
L(x)= (1+o(1))\ c_L/(\log x) ^\upsilon\, ,
\end{equation}
\item[(iii)]  is \emph{super-logarithmic} if  $L$ satisfies 
\begin{equation}\label{fastdef}
 L(x) = (1+o(1)) \ c_L \exp(-( \log x)^{1/\upsilon}) \, ,
\end{equation}
\end{itemize}
where in \eqref{slowdef}-\eqref{fastdef} $x \to \infty$ and  the parameters $\upsilon>1$ and $c_L>0$ can be chosen arbitrarily.
These %respective 
assumptions correspond to asymptotic estimates on the decay of $\tilde L$:
\begin{equation}\label{decaytilde}
\tilde L(x)= (1+o(1)) \times
\begin{cases}
(c_L/(\upsilon-1)) (\log \log x)^{-\upsilon+1} \quad &\text{for {sub-logarithmic} decay},\\
(c_L/(\upsilon-1)) (\log x)^{-\upsilon+1}\quad &\text{for {logarithmic} decay},\\
c_L \upsilon (\log x)^{1-1/\upsilon}\exp(-( \log x)^{1/\upsilon}) \quad &\text{for {super-logarithmic} decay}.
\end{cases}
\end{equation}
\end{definition}

%We present a result that holds without assumptions on $L(\cdot)$ besides its slowly varying character and the fact that $\sum_n L(n)/n=1$ -- in this case Definition~\ref{def:L} is  however helpful to make things more explicit -- and other results that instead are proven only in the framework of Definition~\ref{def:L}. We signal from now that this framework can be generalized without much effort, for example in the direction of going toward higher order iteration of the logarithms, but we made the choice of giving priority to conciseness. 

%For what  follows  $a(\cdot)$, sometimes $\underline{a}(\cdot)$ or $\overline{a}(\cdot)$,  is a  smooth function  from $(0, \overline{\gb})$  to $(0, \infty)$, with the property that $a(\gb)/\gb^2$ is bounded and bounded away from zero when $\gb\searrow 0$. We stress that $a(\gb)$ depends on $K(\cdot)$, that is on $L(\cdot)$,  and on the law of $\go_1$.  In only one case -- the lower bound in Theorem~\ref{th:sharper}(1) -- the interval $(0, \overline{\gb})$ is replaced by $(0, \overline{\gb}/2)$ as a consequence of the use of a second moment estimate. Moreover its value is allowed to change from statement to statement: explicit expressions are given in the proofs.

To state  the results 
we introduce for $\gb< \overline{\gb}$
\begin{equation}
 \label{eq:q1q2}
 q_1(\gb)\, =\, \gb \gl'(\gb)-\gl(\gb)
 \ \ \text{ and } \ \ 
 q_2(\gb)\, :=\,  \gl (2\gb)- 2 \gl(\gb)\, .
 \end{equation}
Note that both quantities are positive if $\gb>0$ and that  $q_1(\gb)< \infty$ for $\gb \in [0, \overline{\gb})$ while
$q_2(\gb)< \infty$ for $\gb \in [0, \overline{\gb}/2)$. On the other hand, 
  $q_1(\gb)= \gb^2/2+ O( \gb^3)$  and $q_2(\gb)= \gb^2+ O( \gb^3)$ for $\gb\searrow 0$. 
We start with a general result -- \textit{i.e.}\ not restricted to the context of Definition~\ref{def:L} -- that says in particular that the transition is of \emph{infinite order}.

\begin{theorem}
\label{th:general}
Consider a general slowly varying function $L$ satisfying $\sum_n L(n)/n=1$.
For every $\gb \in (0, \overline{\gb})$ and every $b \in (0,1)$  there exists $h_0>0$ such that for every $h \in (0, h_0)$
\begin{equation}
\label{eq:general}
\tf(\gb, h) \, \le \, \exp \bigg( - b \, q_1(\gb) \frac 1h \log \bigg( \frac{\tilde L(1/h)}{L(1/h)} \bigg) \bigg)\, .
\end{equation}
\end{theorem}

Note that 
\begin{equation}
\log \left(
\frac{\tilde L(x)}{L(x)}\right) \, =\, (1+o(1)) \times \begin{cases} \log \log x    &\text{for sub-logarithmic and logarithmic decay},\\
\left(\frac{\upsilon -1}{ \upsilon}\right) \log \log x 
 & \text{for super-logarithmic decay}.
\end{cases}
\end{equation}
Therefore, in the framework of Definition~\ref{def:L},  \eqref{eq:general} can be stated as 
\begin{equation}
\label{eq:general-L}
\tf(\gb, h) \, \le \, \exp \bigg( - \bar c_L q_1(\gb) \frac 1h \log \log (1/h) \bigg)\, ,
\end{equation}
with $\bar c_L$ a positive constant that depends only on $L(\cdot)$.

\begin{rem}
The upper bound  \eqref{eq:general-L} holds in much greater, but not in full generality.
 In fact,  consider  the case in which  $\tilde L(x)= C \exp\left(-\int_1^x \eta(u) \dd u /u\right)$, with $C>0$ 
and $\eta(\cdot)$ a positive continuous function vanishing at infinity --~it is straightforward to see that the right-hand side defines a slowly varying function. %and  the slowly varying function Representation Theorem \cite[Th.~1.3.1]{cf:RegVar} 
 Since $\lim_{x \to \infty} \tilde L(x)=0$, we have also that   $\int_1^\infty \eta(u) \dd u /u =\infty$. 
 We then have $L(x)= -x \tilde L'(x)$ so we compute
 \begin{equation}
 \log \frac{\tilde L(x)}{L(x)}\, =\, - \log \eta(x)\, .   
 \end{equation} 
 Therefore, for \eqref{eq:general-L} to hold, it suffices that there exists $c>0$ such that $\log \eta (x) \ge -c \log \log x$ for 
 $x$ large (and necessarily $c\le 1$ otherwise $\int_1^\infty \eta(u) \dd u /u <\infty$). On the other hand, by  choosing $\eta (x)= 1/ \log \log (x+x_0)$,
 $x_0>e$,
 we have  an example to which we can apply \eqref{th:general}, but for which \eqref{eq:general-L} fails. For this example, like for many others 
 that we can write by straightforward generalization, %the  slowly varying  
 $L(\cdot)$  decays even faster than  super-logarithmically.
 \end{rem}

If we restrict to the context of Definition~\ref{def:L} we have much sharper results:

\begin{theorem}
\label{th:sharper}
Choose $L(\cdot)$  as in Definition~\ref{def:L} and fix 
$\gb\in (0, \overline{\gb})$. %(except for the lower bound in the slow case, that is the first inequality in \eqref{eq:sharper-slow} for which $\gb\in (0, \overline{\gb}/2)$)  and every $\gep\in (0,1)$ there exists $h_0>0$ such that 
\begin{enumerate}
\item[(i)] In the sub-logarithmic case we can choose two positive constants $c_+$ and $c_-$ depending only on $\upsilon$ 
so that there exists $h_0$ such that for $h \in (0, h_0)$
\begin{equation}
\label{eq:sharper-slow}
\exp \bigg(- c_- \ q_2(\gb) \frac 1h \log \log (1/h) \bigg) \, \le \, \tf (\gb, h) \, \le \,  \exp \bigg(- c_+ \ q_1(\gb) \frac 1h \log \log (1/h) \bigg)\, .
\end{equation}
\item[(ii)] In the logarithmic case
we can choose two positive constants $c_+$ and $c_-$ depending only on $\upsilon$ 
so that
 there exists $h_0$ such that for $h \in (0, h_0)$
\begin{equation}
\label{eq:sharper-standard}
\exp \bigg(- c_-\ q_1(\gb)\frac 1h \log  (1/h) \bigg) \, \le \, \tf (\gb, h) \, \le \,  \exp \bigg(- c_+\ q_1(\gb) \frac 1h  \log (1/h) \bigg)\, .
\end{equation}
\item[(iii)] In the super-logarithmic case, for $h \searrow 0$ we have
\begin{equation}
\label{eq:sharper-fast}
 \tf (\gb, h) \, = \,  \exp \bigg(- (1+o(1))\left( \frac h {q_1(\gb)} \right) ^{-\upsilon/(\upsilon -1)} \bigg)\, .
\end{equation}
\end{enumerate} 
\end{theorem}

\medskip

We have not made $c_\pm$ explicit in order to keep the statement lighter, but 
they are very explicit (even if probably none is optimal): in \eqref{eq:sharper-slow}
it suffices to choose 
$c_->\upsilon+1$ and  $c_+< \upsilon$, in \eqref{eq:sharper-standard}
it suffices to choose
 $c_->5/2 + \upsilon$ and  $c_+ < \upsilon -1$. Note that the lower bound in \eqref{eq:sharper-slow}
 becomes empty if $\gb> \overline{\gb}/2$, and possibly also for $\gb= \overline{\gb}/2$.
 
Let us also stress that we work here in the framework of Definition~\ref{def:L} for simplicity: some of the expressions simplify thanks to this assumption, but our results can be adapted to a much wider class of slowly varying function $L(\cdot)$. 
We can already point out the only places where Definition~\ref{def:L} is used:
to obtain \eqref{eq:felb1}--\eqref{eq:felb3} in Section~\ref{sec:rs}; in Section~\ref{sec:slowcase}; to obtain \eqref{eq:Aslow}--\eqref{eq:Afast} in Section~\ref{sec:iub}.

 \subsection{About our methods, about perspectives and overview of the work}
Some important remarks are in order:

\smallskip
\noindent
{(1)}  We stress once again the partial agreement with respect to the physical literature. 
 Theorem~\ref{th:general} establishes that the transition is of infinite order, but it
 is not in agreement with the works that predict a behavior of the type $\exp(-c/h)$, with $c$ a constant. 
Note that $c$ is rather a diverging function, and Theorem~\ref{th:sharper} gives more information: the divergence is only slowly varying for the sub-logarithmic and logarithmic cases, but it  behaves as a power law in the super-logarithmic case. 

\smallskip
\noindent
{(2)}
 We note that the result we find go in the opposite direction 
with respect to the   $\exp(-c/\sqrt{h})$ behavior of \cite{cf:DR,cf:tangchate}: there is no disagreement between  
our results and the claim in \cite{cf:DR,cf:tangchate} because  \cite{cf:DR,cf:tangchate} are about pinning models. 
Nevertheless this is very  intriguing and calls for deeper understanding. With respect to this,
we recall that the strong disorder RG analysis in \cite{cf:dfisher}  is inspired and built on the results of McCoy and Wu  \cite{cf:MW}
on the two dimensional Ising model with columnar disorder. McCoy and Wu
predict an  infinite order transition, with a precise form of the singularity that is different from the exponential type of
essential singularity  found or predicted for copolymer and pinning. 

\smallskip
\noindent
{(3)} Our results definitely exploit the fact that the critical point of the quenched system is explicitly known and, considering also \cite{cf:GL,cf:FF2},  this appears to be for the moment 
an unavoidable ingredient.   Also  in the case of \cite{cf:DR} the critical value is exactly known and this is exploited in the analysis of the model. 
It would be of course a major progress if methods could be developed to deal with cases in which the critical point is known only 
implicitly. And this appears as a necessary step if one wants  to understand the criticality of the copolymer model for $\ga>0$  
or if one wants to really apprehend % in a mathematical way 
the strong disorder universality class.

\medskip

The rest of the paper is devoted to the proofs.

\begin{itemize}
\item In Section~\ref{sec:rs}, we prove a general lower bound based on revisiting the rare stretch strategy in the direction of making it sharper. 
Essentially, we exploit local limit results   instead of rougher  integral limit 
results like in previous  works \cite{cf:BoGi} and \cite[Sec.~6.2 and Sec.~5.4]{cf:GB}. This covers the lower bounds in  \eqref{eq:sharper-standard} and \eqref{eq:sharper-fast}.
\item In Section~\ref{sec:slowcase} we provide the lower bound in \eqref{eq:sharper-slow}: this requires a new, more entropic strategy with respect to the rare stretch strategy and it is based on applying the second moment method
on a suitably trimmed partition function. Of course here the difficulty is in finding a suitable  subset of the renewal trajectories that yield 
a contribution to the partition function that is sufficient to match the upper bound and that is not too wide so that the second moment can be compared with the square of the first moment on sufficiently large volumes.  
\item In Section~\ref{sec:ub} we provide the proof of Theorem~\ref{th:general}. This is achieved via a penalization argument inspired by the analogous bound used for the two dimensional free field in \cite{cf:FF2}. The strategy is sketched at the beginning of Section~\ref{sec:ub}.
\item In Section~\ref{sec:iub} we provide the proof of the upper bounds in Theorem~\ref{th:sharper}. 
The key idea here is to pass from a global to a targeted penalization: we avoid penalizing the 
charges in regions that are not visited. In practice this requires setting up an appropriate coarse graining procedure that builds on        the basic structure 
of the argument in \cite{cf:DGLT}. 
Therefore, with respect to the case of  the two dimensional free field   \cite{cf:FF2}, we are  able to upgrade the upper bound penalization procedure to get to results that are \emph{optimal} (in a sense   that can be read out of Theorem~\ref{th:sharper}). 
\end{itemize}
\smallskip

We use the notation $ \lint n,m \rint:= [n,m]\cap \bbZ$, for $n, m \in \bbZ$ and $n\le m$.

\section{The lower bound: rare stretch strategy} 
\label{sec:rs}
 
We develop in this section   a more quantitative version of the 
 rare stretch strategy argument in \cite[Sec.~6.2 and Sec.~5.4]{cf:GB}.  This relatively simple argument yields \emph{optimal} bounds in the logarithmic and super-logarithmic cases.

Recall the definition of $q_1(\gb)$  %and $q_{2}(\gb)$ 
in \eqref{eq:q1q2}.
\begin{theorem}
\label{th:rss}
For $\gb \in (0, \overline{\gb})$ we have:
\begin{itemize}
\item [(i)] For every  $b>7/2$ in the sub-logarithmic case and for every $b>5/2+\upsilon$ in the logarithmic case there exists $h_0>0$ such that for every $h\in (0, h_0)$
\begin{equation}
\tf(\gb,h) \, \ge\, \exp \bigg( -b\,  q_1(\gb) \frac{\log(1/h)}h \bigg)\, ;
\end{equation}

\item [(ii)] In the super-logarithmic case, for every $b>1$ there exists $h_0>0$ such that for every $h\in (0, h_0)$
\begin{equation}
\tf(\gb,h) \, \ge\, \exp \bigg( -b \left( \frac{q_1(\gb)}{h} \right)^{- \frac \upsilon{\upsilon-1}} \bigg)\, .
\end{equation}
\end{itemize}
\end{theorem}

The argument is not optimal in the sub-logarithmic case, in particular
 because in that setup the entropic cost for having renewal points only at multiples of $\ell$ becomes non-negligible.
 \medskip

\noindent
{\it Proof.}
 Recall the assumptions on $\go$, cf. \eqref{defgl}-\eqref{normalize}. 
A  sharp version of the Cram\'er's Large Deviation principle says that for every $x \in (0, C)$, 
 $C:= \lim_{\gb\nearrow \bar \gb} \gl'(\gb)\in (0,\infty]$, there exists 
$c(x)>0$ such that for 
\begin{equation}
\label{eq:sharpLD}
\bbP \bigg( \sum_{j=1}^n \go_j \,\ge \, n x \bigg) \stackrel{n \to \infty}\sim \frac{c(x)}{n^{1/2}} \exp \big( -n \gS(x) \big)\, ,
\end{equation}
where $\gS(x):= \sup_y [ xy - \gl(y)]$. The value of $c(x)$ is irrelevant for us, but it can be  found in \cite[Th.~1 and Th.~6]{cf:Petrov} along with a proof of \eqref{eq:sharpLD}.

 Now we  fix $\gb \in (0, \overline{\gb})$ and we define  $q:= \gl'(\gb)$: note that $q$ solves the conjugate Legendre problem $\gl(\gb)= \sup_x (\gb x - \gS(x))$, that is $\gl(\gb)= \gb q - \gS(q)$. Note that $q_1(\gb)=\gS(\gl'(\gb))$.
We choose a (sufficiently large) integer $\ell$ and for $j=0,1, \ldots$ we set
\begin{equation}
\label{eq:Ej}
E_j\,:=\, \bigg\{\go: \, \sum_{n=1}^\ell \go_{j \ell+n}\, \ge \, q \ell \bigg\}\, , 
\end{equation} 
so   that  $\{\ind_{E_j}\}_{j=0,1, \ldots}$ is an IID sequence of Bernoulli random variables of parameter $p(\ell)$, where
\begin{equation}
\label{eq:pell}
p(\ell)\, :=\, \bbP\left( E_1\right)\, \in \, \left(\frac{c(q)}{2\ell^{1/2}} \exp \left(-\ell \gS(q)\right), \frac{2c(q)}{\ell^{1/2}} \exp \left(-\ell \gS(q)\right) \right) \, .
\end{equation}
and the inclusion statement follows from 
\eqref{eq:sharpLD} for $\ell$ greater than a suitably chosen value that depends on $q$.
Associated   with  such a sequence we introduce the sequence of IID geometric random variables
$\{G_j\}_{j=1,2, \ldots}$, $G_i\in \{0,1, \ldots\}$, that count the number of failures before the  $i^{\textrm{th}}$ success in the Bernoulli sequence $\{\ind_{E_i}\}_{i= 0,1, \ldots}$, defined recursively as follows
\begin{equation}\
G_j:= \min \bigg \{ k = 0,1, \ldots \ : \sum_{i=0}^k \ind_{E_i}(\go)= j \bigg \} - \bigg(\sum_{m=1}^{j-1} G_m \bigg)  -j\, ,
\end{equation}
where the sum from $m=1$ to $j-1$ should be read as zero for $j=1$.
We call $\cN_N(\go)$ the  number of successes up to epoch $N$, that is 
$\cN_N(\go) = \sum_{j\le N/\ell} \ind_{E_j}(\go)$. Alternatively, $\cN_N(\go) = \sup \{k:\, \sum_{i=1}^k(G_k+1)\ell \le N\}$.
 By the law of large numbers,  $\bbP(\dd \go)$ almost surely we have 
 \begin{equation}
 \lim_{N \to\infty} \frac{\cN_N(\go)}N \,= \,\frac{p(\ell)}\ell \, .
 \end{equation}
The rare stretch lower bound strategy is based on selecting,
 once the disorder $\go$ is known, only the renewal and excursion trajectories $(\tau, \gD) \in \gO_{N, q}(\go)$
 for which the process stays above the interface  except in the $\cN_N(\go)$ successful $\ell$-blocks
 \begin{equation}
 \gO_{N, q}(\go) :=\, \bigg \{(\tau, \Delta):\,  \{  n\in \lint 1,N \rint  : \,   \Delta_n=1\} =  %\!\!\!\!\!\!  
 %\uniontwo{i \le N/\ell-1  :}{\ind_{E_i}(\go)=1}  \lint i\ell+1, (i+1)\ell \rint
 \union_{i \le \frac N\ell-1:\, \ind_{E_i}(\go)=1}  \lint i\ell+1, (i+1)\ell \rint
\bigg \}\, .
\end{equation}
 % (see \cite[Sec.~6.2 and Sec.~5.4]{cf:GB} for a formal definition).
In order to find a lower bound for  $\bP \left(\gO_{N, q}(\go) \right)$ it is sufficient to consider the strategy where the renewal only targets the extremities of the intervals $\lint i\ell+1, (i+1)\ell \rint$ for which $\ind_{E_i}(\go)=1$ plus the last point $N$, and choses the sign of each excursion adequately.
 We invite the reader to look at Figure~\ref{fig:1}. We obtain thus
 \begin{multline}
 \log \bP \left(\gO_{N, q}(\go) \right) \, \ge \, 
 \sum_{j=1}^{\cN_N(\go)} \log \left( \frac 12K\left(G_j \ell\right) \right)\ind_{\{G_i\ge 1\}}
 +
 \cN_N(\go) \log \left( \frac 12 K(\ell) \right)- 2 \log N\, ,
 %\\ + \log\left(\frac{1}{2}K(N- H_N)\right)\ind_{\{N- H_N\ge 1\}}\, ,
 \end{multline} 
 %where $$H_N(\go):= \ell \left(\max\{ i\le N/\ell-1 \ : \ \ind_{\bE_{i}}(\go)=1  \}+1\right).$$
 where the term $- 2 \log N$ comes from a rough estimate of the last excursion, which could be empty 
 and in any case it is not longer than $N$, hence for large $N$ its contribution is bounded from below by
 $\log (K(N)/3)$, and $1/3$ is used instead of the probability $1/2$ of being above the interface because 
 $K(\cdot)$ is asymptotically equivalent to a decreasing function but it is not necessarily decreasing (or non increasing), see 
 \cite[Th.~1.5.3]{cf:RegVar}.

\begin{figure}[htbp]
\centering
\includegraphics[width=14 cm]{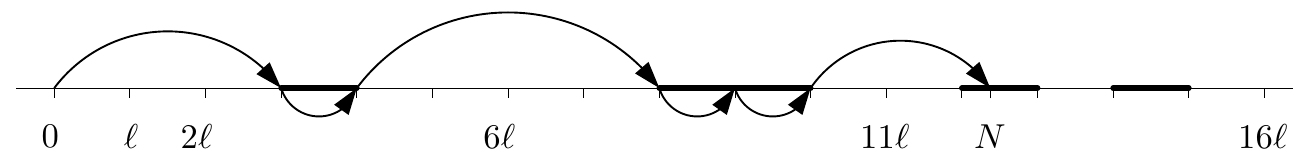}
\vskip-.2cm
\caption{\label{fig:1} 
The system is partitioned into blocks of length $\ell$ and the $j^{\textrm{th}}$ block is successful 
if the average of the $\go_j$'s in the block is at least $q>0$: these blocks are marked by a thick line and we see $5$ of them  in the example we consider. The variables $G_j$'s recursively count the unsuccessful blocks before a successful one. So for this illustration case, $G_1=3$, $G_2=4$, $G_3=0$, $G_4=2$, $G_5=1$ and
$G_6>1$. The excursions we choose are indicated by the arrows: they begin and terminate at the extremity of a block, except possibly the last one which ends in $N$ which is not necessarily a multiple of $\ell$. The excusions below the interface are always of length $\ell$ and they correspond to the $\ell$-blocks that are subset of $[0, N]$. All the other excursions are above.
}
\end{figure}

% $K(0):=1$: previously $K(0)$ was undefined and we use this definition only in this proof to take care of the 
%  event of having contiguous successful blocks (see Figure ADD and its caption). 
 Therefore using the convention $K(0):=1$, we obtain that  $\bbP(\dd \go)$ almost surely
 \begin{equation}
 \liminf_{N \to \infty} \frac 1N \log  \bP \left(\gO_{N, q}(\go) \right) \, \ge \, \frac{p(\ell)}\ell
\Big( \bbE \left[ \log K\left(G_1 \ell\right) \right]+ \log K(\ell) -2 \log 2 \Big)\, .
 \end{equation}
 In order to estimate the term  $\bbE \left[ \log K\left(G_1 \ell\right) \right]$, we consider the variable $g_1:= p(\ell)G_1$.
By the Potter bound \cite[Th.~1.5.6]{cf:RegVar} for every $a>0$ there exists $b>0$ such that
$L(y)/L(x)\le e^b \max((y/x)^a,(x/y)^a)$ for every $x,y\ge 1$, so we see that
\begin{equation} 
\label{difrence}
\begin{split}
 \left| \log K\left(G_1 \ell\right)- \log K\left(\ell /p(\ell)\right)\right| \, &=\, \left| \log K\left(g_1 \ell /p(\ell)\right)- \log K\left(\ell /p(\ell)\right)\right|
 \\ &\le\, 
   \left(b+(1+a) \vert \log g_1\vert\right) \ind_{\{g_1>0\}}+   2\log (\ell/p(\ell)) \ind_{\{g_1=0\}}\, ,
   \end{split}
 \end{equation}
and one directly
checks that the expectation of the right-hand side  in \eqref{difrence} is uniformly bounded in $\ell$.

%\\
%&=\, 
%\frac{p(\ell)}\ell
%\left( \bbE \left[ \log K\left(G_1 \ell\right)\vert G_1>0 \right](1-p(\ell))+ \log K(\ell) -2 \log 2 \right)
%\,,
%\end{split}

%  and, conditioned to $G_1>0$,
%  $G_1$ just becomes, in law, equal to $G_1+1$: we call $G'_1$ this new geometric variable. From now on we restrict ourselves to the
% framework of Definition~\ref{def:L} and we call $L_\infty(\cdot)$ the smooth function that identifies 
% the asymptotic behavior of $L(\cdot)$ (and we set $c_L$ equal to one). 
% Note that the inter-arrival distribution is evaluated only for large arguments, i.e. at least $\ell$, and we can therefore use its asymptotic behavior 
% at the expense of a constant multiplicative error. So:
%\begin{equation}
% \liminf_{N \to \infty} \frac 1N \log  \bP \left(\gO_{N, q}(\go) \right) \, \ge \,
% \frac{p(\ell)(1-p(\ell))}\ell
%\left( \bbE \left[ \log K_\infty \left(G'_1 \ell\right) \right]- \log \ell ) \right)\, ,
% \end{equation}
% with $K_\infty(x)=L_\infty(x)/x$.
% It is practical at this point to use the convexity of $\log K_\infty(\cdot)$ (see Remark~\ref{rem:convex}), so
% $\bbE \left[ \log K_\infty\left(G'_1 \ell\right) \right]$ is bounded below by 
%$\log K_\infty\left( \ell/p(\ell)\right)$. So, by exploiting \eqref{eq:pell}, we obtain
In what follows $\gep>0$ 
 is an arbitrarily small constant that may change from line to line.
For $\ell$ sufficiently large
\begin{equation}
\begin{split}
 \liminf_{N \to \infty} \frac 1N \log  \bP \left(\gO_{N, q}(\go) \right) \, &\ge \,
  \frac{p(\ell)}{\ell} \left(
  \log \frac{p(\ell)}\ell + \log L\left(\frac\ell{p(\ell)}\right)+ \log L(\ell)- \log \ell - 2 \log 2
  \right)
% \\
 % &\ge \,\frac{p(\ell)}{\ell} \left(\log(p(\ell)\ell^{-2}) +\log L\left(\ell p(\ell)^{-1}\right) L(\ell) -C\right)
 \\
 &\ge\,   \frac{p(\ell)}{\ell} \left(-\ell \gS(q)  +\log L\left(\frac 1 {p(\ell)}\right) -(5/2+\gep) \log \ell\right)
 \, ,
 \end{split}
 \end{equation}
 where the first inequality is obtained by replacing $G_1$ with $p(\ell)^{-1}$ by using \eqref{difrence}.
 The second inequality is obtained by using \eqref{eq:sharpLD} and by using again the Potter bound
 to replace $L(\ell/p(\ell))$ with $L(1/p(\ell))$ and to neglect $\log L(\ell)$.
 From this estimate 
 we readily obtain that $\tf(\gb, h)$ is bounded below by 
 \begin{align}
 \label{eq:estfor3}
 \liminf_{N \to \infty}
 \frac 1N &\log Z_{N, h} \big( \gO _{N, q}(\go) \big)\, \ge \, 
 \ell ( \gb q-\gl(\gb)+h)
 \lim_{N \to \infty} \frac{\cN(\go)}N +\liminf_{N \to \infty} \frac 1N \log  \bP \big(\gO_{N, q}(\go) \big)
\notag  \\
& \ge \,  \frac{p(\ell)}{\ell} \Big( \left(\gb q  -\gl(\gb) - \gS(q)+h\right) \ell+ \log L(p(\ell)^{-1})- (5/2+\gep) \log \ell
 \Big)
\notag  \\
 &\ge \,  \frac{p(\ell)}{\ell} \Big( h \ell-  (5/2+\gep)\log \ell +\log L(p(\ell)^{-1}) \Big)\, =: \,  \frac{p(\ell)}{\ell} g(h,\ell)\, ,
 \end{align} 
 where in the last line  we have used that $q$ is the optimizer of the (conjugate) Legendre problem and the very last step defines 
 $g(h,\ell)$.
 
 \smallskip
 
 It is now a matter of going through the three classes of slowly varying functions that we consider and see how large $\ell=\ell(h)$
 has to be chosen in order to guarantee a positive $g(h, \ell(h))$ gain. In what follows $h$ is chosen small, possibly smaller than a constant that depends also on $\gep$. Moreover $C>0$ is a constant that in all cases can be easily identified:
 
 \smallskip
 
 \begin{itemize}[leftmargin=*]
 \item[--] In the super-logarithmic case we have
 \begin{equation}
 \log L\left(p(\ell)^{-1}\right)\,\ge \, - \left(    \ell \gS(q)+ (\log \ell)/2 + C\right)^{1/\upsilon} \, \ge \, 
 -(1+ \gep) \ell^{1/\upsilon} \gS(q) ^{1/\upsilon}\, ,
 \end{equation}
 and we readily see that if $\ell \ge (1+ \gep)h^{-\upsilon/(\upsilon-1)}  \gS(q) ^{1/(\upsilon-1)}$ then there exists $c(\gep, \upsilon)>0$ such that
 $g(h, \ell)\ge c(\gep, \upsilon) h\ell$. Therefore \eqref{eq:estfor3} yields 
 \begin{equation}
 \label{eq:felb1}
 \tf(\gb, h) \, \ge \, c(\gep, \upsilon) h p(\ell) \, \ge \, \exp\left( - (1+\gep) h^{-\upsilon/(\upsilon-1)}  
 \gS(\gl'(\gb))^{\upsilon/(\upsilon-1)}\right)\, ,
 \end{equation}
 where in the last step we have used the explicit expression for $q$.
\item[--] In the logarithmic case
 \begin{equation}
 \log L\left(p(\ell)^{-1}\right) \ge  - \upsilon \log \big( \ell \gS(q) \big) -\frac  1 2\upsilon \log \log \ell  -C
 \, \ge \, -(1+ \gep) \upsilon \log \ell\, ,
 \end{equation}
 so that
 \begin{equation}
 g(h, u) \, \ge\, h\ell -\left( 5/2+v+\gep\right)  \log \ell\, ,
 \end{equation}
 and it suffices to choose $\ell \ge (5/2+\upsilon+2\gep) \log (1/h) /h$ to have 
 $g(h, u)\ge c(\gep, \upsilon) h \ell$. Therefore 
  \begin{equation}
   \label{eq:felb2}
 \tf(\gb, h) \, \ge \, c(\gep, \upsilon) h p(\ell) \, \ge \, \exp\left( -(5/2+\upsilon+\gep) \gS(\gl'(\gb)) \frac{\log (1/h)}h 
\right)\, .
 \end{equation}
\item[--] Last, in the sub-logarithmic case the computation is the same as in the logarithmic case because the iterated logarithm  in  the asymptotic of $L(x)$ is 
 irrelevant for the purpose of this computation and $L(x)$ can be replaced, with an error that can be hidden by $\gep$,
 by $1/\log (x)$, that is the logarithmic case with $\upsilon=1$. Hence the net result is 
\begin{equation}
   \label{eq:felb3}
 \tf(\gb, h) \, \ge \, c(\gep, \upsilon) h p(\ell) \, \ge \, \exp\left( - \left(7/2+\gep\right) \gS(\gl'(\gb)) \frac{\log (1/h)}h 
\right)\, .
 \end{equation} 
 \end{itemize}
 \smallskip
 
 The proof is complete once we recall that $\gS(\gl'(\gb))=
 \gb \gl'(\gb)-\gl(\gb)=q_1(\gb)$. 
 \qed
 
% \medskip
% 
% \begin{rem}
% \label{rem:convex}
% The convexity of $x\mapsto \log K_\infty(x)$, $x$ larger than a suitably chosen positive constant, follows by direct checking. Note that
% $L_\infty (x)$ can be written, possibly up to a multiplicative constant, as 
% $\exp(-\int_{x_0}^x (\epsilon(t)/t) \dd t$, with a proper choice of a smooth positive function
% $\epsilon(\cdot)$ and $\epsilon(x)=o(1)$. The claimed convexity boils then down to 
% showing that $-\epsilon'(x)+(1+\epsilon(x))/x \ge 0$. In particular it suffices to verify that
% $\epsilon(x)= x (\log L_\infty(x))'$ is decreasing. In our three cases $\epsilon(x)$ is: $(1+\upsilon/\log \log x)/\log x$,
% $\upsilon/\log x$ and $(1/\upsilon)(1/\log x)^{1-1/\upsilon}$.
% \end{rem}
% 
 \section{Improved lower bound in the sub-logarithmic  case}
 \label{sec:slowcase}
 
 Recall the definition \eqref{eq:q1q2} of
$ q_2(\gb)\, :=\,  \gl (2\gb)- 2 \gl(\gb)$,
 which is positive for $\gb>0$ and finite for $ \gb < \overline{\gb}/2$.  
 
 \begin{theorem}
 \label{th:slowcase}
 Assume that $L(\cdot)$ satisfies \eqref{slowdef} and that $q_2(\gb) < \infty$. Then for every $c> \upsilon+1$ there exists $h_0>0$ such that  for every $h \in (0, h_0)$
 \begin{equation}
 \label{eq:slowcase}
 \tf(\gb, h) \, \ge \, \exp\left( -c\, q_2(\gb) \frac{\log \log (1/h)}{h}\right)\, .
 \end{equation}
 \end{theorem}

%We work with $\tilde L(x)\sim c_L(\log \log (x))^{-\upsilon+1}/(\upsilon -1)$.

\noindent
\emph{Proof.}
We introduce 
\begin{equation}
\label{parameterz}
k = k(h) :=  c_1 \frac{\log \log (1/h)}{h}\, , \ \  M := \exp\left( c_2 k\right),
\ \ N=M^{2} (\log M)^3\
\text{ and } m= \frac N {M^2\log M}\, ,
\end{equation}
with $c_1$ and $c_2$ positive constants that will be chosen later on. At times we will also do as if $k$, $M$ and $m$ were integers.
We introduce the event
\begin{multline}
\label{eq:Gm}
G_m\, :=\, \Big\{(\tau, \gD):\, 
\tau_{2i+1}-\tau_{2i} \in [M, M^2] \text{ and }
\tau_{2i+2}-\tau_{2i+1} \in [1,k] \text{ for }  i \in \lint 0, m-1\rint,\\
\gD_{\tau_{j}+1}= \frac{1-(-1)^j}2\text{ for } j\in \lint 0,  2m\rint  \text{ and }
 \tau_{2m+1}=N
\Big\}\, .
\end{multline}
In words, we focus on the trajectories that make alternatively a {\sl long} and a {\sl short} excursion, and this is repeated a fixed number of times ($m$) before reaching the point $N$: long jumps do not collect any energy
 ($\gD$ in the excursion is zero), while the short ones do ($\gD$ in the excursion is one). 
The last excursion is long, hence it does not collect any energy, but we stress that it is very long: in fact the $2m$ excursions that precede 
the last excursion add up at most to $m(M^2+k)= (1+o(1)) N/ \log M$, and therefore the length of the last excursion is  asymptotically equivalent to the size of the system as $h \searrow 0$.
 We then consider  the partition 
function restricted to these trajectories:
\begin{equation}
\tilde Z_{N, \go}\, :=\, Z_{N, \go}\left( G_m \right)\, . 
\end{equation}
We have 
\begin{align}
\label{eq:PZ}
\bbE \log Z_{N, \go} \, & =\, \bbE \Big[ \log Z_{N, \go}; \, \tilde Z_{N, \go} \ge \tfrac 12 \bbE \tilde Z_{N, \go} \Big]
+
\bbE \Big[ \log Z_{N, \go}; \, \tilde Z_{N, \go} < \tfrac 12 \bbE \tilde Z_{N, \go} \Big]
\notag \\
 &\ge  \bbP \Big(\tilde Z_{N, \go} \ge \tfrac 12 \bbE \tilde Z_{N, \go} \Big)  \log \Big( \tfrac12 \bbE\tilde Z_{N, \go} \Big)+ \bbP \Big(\tilde Z_{N, \go} < \tfrac 12 \bbE \tilde Z_{N, \go} \Big) \log \big( K(N)/2 \big) \notag\\
&\ge  \frac{\bbE \big[\tilde Z_{N, \go}\big]^2}{4\bbE \big[ (\tilde Z_{N, \go})^2 \big]}  \log \Big( \tfrac12 \bbE\tilde Z_{N, \go} \Big)
 - 2 \log N.
\end{align}
where in the first inequality we used $Z_{N, \go}\ge K(N)/2$ (that holds in full generality) and in the second one
we have applied the Paley-Zygmund inequality and we have assumed that 
$ \log( (1/2) \bbE\tilde Z_{N, \go} )\ge 0$, that is 
$\bbE \tilde Z_{N, \go}\ge 2$. 

To conclude we need to estimate the first two moments of $\tilde Z_{N, \go}$.

\subsubsection{First moment estimate}
We have
\begin{equation}
\label{eq:firstmom}
\begin{split}
\bbE \tilde Z_{N, \go}\,&=\, 
\bE \bigg[
\exp \Big( h \sum_{i=1}^m (\tau_{2i+2}-\tau_{2i+1}) \Big); \, G_m
\bigg]
\\&\ge \, \bigg(\frac 12 \sum_{n=M}^{M^2} K(n)
\bigg)^m \bigg( \frac 12 \sum_{n=1}^k e^{hn}K(n)\bigg)^m \frac{K(N)}3\, ,
\end{split}
\end{equation}
where the inequality comes from the last excursion which is shorter than $N$ ($3$ is present instead of $2$ because $K$ is  monotonic only in an asymptotic sense).
It follows from our sub-logarithmic decay  assumption (recall \eqref{decaytilde}) and our choice of parameters \eqref{parameterz} that when $h$ tends to zero
\begin{equation}
\label{eq:asympt0.1}
\sum_{n=M}^{M^2} K(n)\, =\, \tilde L(M)-\tilde L(M^2)=  \frac{c_L\log 2}{\left\vert \log h\right\vert^\upsilon}(1+o(1)).
\end{equation}
For the second term instead  we split   $\sum_{n=1}^k e^{hn}K(n)$
according to whether $n$ is smaller or larger than $1/h$ and we see that the first 
sum is bounded by $e$. For the second sum, we remark that we have for $h$ sufficiently small 
\begin{multline}
\label{eq:asympt0.2}
 \sum_{n= \lfloor 1/h\rfloor}^{ k}  e^{hn}K(n)\, =\, e^{hk} \sum_{j=0}^{k- \lfloor 1/h\rfloor} e^{-hj} \frac{L(k-j)}{k-j}\\
  \ge  e^{hk}
 \frac{L(k)}{hk}(1-\gep) \ge  \frac{c_L \vert \log h \vert ^{c_1-1}}{c_1 (\log \log (1/h))^{\upsilon+1}}(1-2\gep)  \, .
\end{multline}
The first inequality can be achieved by considering the  sum restricted to $j\le \sqrt{k/h}$
(or anything that tends to infinity faster than $1/h$ but much slower than $k$) and observe that, uniformly for these values of $j$, we have 
$\frac{L(k-j)}{k-j}=(1+o(1)) \frac{L(k)}{k}$.
The last step is a consequence of the sub-logarithmic decay assumption and of  \eqref{parameterz}.
The interested reader can check that both inequalities correspond to asymptotic equivalences.

It is now sufficient to observe  that \eqref{eq:firstmom}-\eqref{eq:asympt0.2} readily imply that 
for every $\gep>0$ we can find $h_0>0$ so that we have 
\begin{equation}
\label{eq:addGB1}
\bbE \tilde Z_{N, \go}\,\ge \, \vert \log h \vert ^{c_1-\upsilon -1-\gep}\, ,
\end{equation}
for every $h\in (0, h_0)$. Therefore 
 if 
\begin{equation}
\label{eq:condc1}
c_1\, >\, \upsilon+1\, , 
\end{equation} 
 we have  in particular, possibly redefining $h_0$, that 
$\bbE \tilde Z_{N, \go}\ge 2$ (recall that the last step in \eqref{eq:PZ} was made under this assumption).

\subsubsection{Second moment estimate} 
We  aim at showing that for $h \searrow 0$
\begin{equation}
 \label{eq:addGB2}
 \frac{\bbE \big[
 (\tilde Z_{N, \go}
 )^2
 \big]}{\bbE \big[
\tilde Z_{N, \go}
 \big]^2}\, =\, 1 +o(1)\, ,
\end{equation} 
and we 
start by observing that
 \begin{equation}
 \label{eq:startZ2}
 \frac{\bbE \big[
 (\tilde Z_{N, \go}
 )^2
 \big]}{\bbE \big[
\tilde Z_{N, \go}
 \big]^2}\, =\, \tilde \bE^{\otimes 2}_h \bigg[
 \exp \bigg( q_2(\gb) \sum_{n=1}^N \gD^{(1)}_n \gD^{(2)}_n \bigg)
 \bigg]\, ,
 \end{equation}
where $q_2(\gb)$ is given in \eqref{eq:q1q2} and  the measure $\tilde \bP_h$ is defined by
\begin{equation}
\frac{\dd \tilde \bP_h}{\dd \bP}(\tau, \gD):= \frac{1}{\bbE \big[
\tilde Z_{N, \go}
 \big]} \exp \bigg( h \sum_{i=1}^m |\tau_{2i}-\tau_{2i-1}| \bigg) \ind_{G_m}(\tau, \gD)\, .
\end{equation}
In order to obtain an upper bound for the right hand side of \eqref{eq:startZ2}, we are going to prove a bound for the expectation with respect to $\gD^{(1)}$ which is uniform for every realization of  $\gD^{(2)}$ in $G_m$.
Let $A$ denote an arbitrary set of the form 
\begin{equation}\label{leseta}
A:= \bigcup_{i=1}^\infty \lint a_i,b_i\rint, 
\end{equation}
where $a_1\ge M$ and for all $i= 1,2, \ldots$
\begin{equation}\label{spacingz}
(b_i-a_i)\le k \quad \text{ and } \quad  b_{i+1}-a_{i}\ge M.
\end{equation}
 Note that it would be more natural to consider the union to stop at $m$ but it is more practical for us to consider  here an infinite set.
To conclude it is sufficient  to show that, uniformly over all possible sets $A$
\begin{equation}
\label{eq:nts3}
 \tilde \bE_h \bigg[
 \exp \bigg( q_2(\gb)\sum_{i=1} ^{m} \big \vert \lint \tau_{2i-1}+1, \tau_{2i} \rint \cap A \big \vert \bigg)
 \bigg] \stackrel{h \searrow 0}= 1+ o(1)\, .
 \end{equation}
Our first step is to replace  $\tilde \bP_h$ by a nicer measure under which jumps are independent.
For $m\in \bbN$ let $\hat \bP^m_h$ be a measure on a finite sequence $\{\tau_i\}_{i\in \lint 0, 2m\rint}$ for which $\tau_0=0$,  the increments are independent and for which 
for every $i\in \lint 0,m-1\rint$
\begin{equation}
\label{eq:indeps}
\begin{split}
\hat \bP^m_h \left(\tau_{2i+1}-\tau_{2i}= n\right)&= \frac{K(n)\ind_{n\in [M,M^2]}}{\sum_{n=M}^{M^2} K(n)},\\
\hat \bP^m_h\left(\tau_{2i+2}-\tau_{2i+1}= n\right)&= \frac{e^{hn} K(n)\ind_{n\in [1,k]}}{\sum_{n=1}^k e^{kn} K(n)}.
\end{split}
\end{equation}
Analogously we define also the infinite random sequence $\{\tau_i\}_{i=0,1, \ldots}$
and its law is denoted by  $\hat \bP_h$. We set 
\begin{equation}
\label{eq:Gmtilde}
\tilde G_m\, :=\, 
 \left\{\tau:\, 
\tau_{2i+1}-\tau_{2i} \in [M, M^2] \text{ and }
\tau_{2i+2}-\tau_{2i+1} \in [1,k] \text{ for }  i \in \lint 0, m-1\rint
\right\}\, .
\end{equation}
We will commit abuse of notation in using the same notations for random variables and standard, \textit{i.e.}\ non-random, variables.

\begin{lemma}
\label{th:hublemma}
Let $\tilde \bP^m_h$ be the distribution of $\{\tau_i\}_{i\in \lint 0, 2m\rint}$ under $\tilde \bP_h$.
Then  we have 
%\begin{equation}
%\left\| \frac{\dd \tilde \bP^m_h}{\dd \hat \bP^m_h} \right\|_{\infty}\le \frac{\max_{n\in \lint 0, m (M^2+k)\rint } K(N-n)}{\min_{n\in \lint 0, m (M^2+k)\rint} K(N-n)}\stackrel{h \searrow 0}=1+o(1).
%\end{equation}
\begin{equation}
\label{eq:hublemma}
\sup_{\{\tau_i\}_{i\in \lint 0, 2m\rint}\in \tilde G_m}
\frac{\dd \tilde \bP^m_h}{\dd \hat \bP^m_h} \left(\tau_0, \tau_1, \ldots, \tau_{2m}\right)\stackrel{h \searrow 0}=1+o(1)\, .
\end{equation}
\end{lemma}

\medskip 

\begin{proof}
$\hat \bP^m_h(\tau_0, \tau_1, \ldots, \tau_{2m})$ can be directly expressed from \eqref{eq:indeps}, but it is helpful
to write it in the more implicit fashion:
\begin{equation}
\label{eq:denom}
\frac{\prod_{i=1}^m K(\tau_{2i-1}-\tau_{2i-2})K(\tau_{2i}-\tau_{2i-1})e^{h(\tau_{2i}-\tau_{2i-1}) }
}
{
\sum_{\tau'\in \tilde G_m}\prod_{i=1}^m K(\tau'_{2i-1}-\tau'_{2i-2})K(\tau'_{2i}-\tau'_{2i-1})e^{h(\tau'_{2i}-\tau'_{2i-1}) }
}\, .
\end{equation}
On the other hand $\tilde \bP^m_h(\tau_0, \tau_1, \ldots, \tau_{2m})$ instead is
\begin{equation}
\label{eq:num}
\frac{\prod_{i=1}^m K(\tau_{2i-1}-\tau_{2i-2})K(\tau_{2i}-\tau_{2i-1})e^{h(\tau_{2i}-\tau_{2i-1}) } K(N- \tau_{2i})
}
{
\sum_{\tau'\in \tilde G_m}\prod_{i=1}^m K(\tau'_{2i-1}-\tau'_{2i-2})K(\tau'_{2i}-\tau'_{2i-1})e^{h(\tau'_{2i}-\tau'_{2i-1}) }
K(N- \tau'_{2i})} \, ,
\end{equation}
and the denominator of this expression cannot be easily simplified. 
But for every $\tau\in \tilde G_m$ we have that $N-\tau_{2m}\ge N-m(M^2+k)\ge N(1-2/ \log M)$. Therefore  
we have that the ratio
\begin{equation}
\frac{K(N- \tau_{2i})}{K(N- \tau'_{2i})} \,=\,  1+o(1)\, ,
\end{equation}
for $h \searrow 0$,
uniformly in $\tau$ and $\tau'$ in $\tilde G_m$.
But the ratio of \eqref{eq:num} and \eqref{eq:denom}
can be bounded precisely by this ratio and therefore Lemma~\ref{th:hublemma} is proven.
\end{proof}

To conclude we are going to show that for $A$ satisfying \eqref{leseta} and any $j$
\begin{equation}
\label{sqwash}
 \hat \bE_h \bigg[
 \exp \bigg( q_2(\gb)\sum_{i=1} ^{j} \big| \lint \tau_{2i-1}+1, \tau_{2i} \rint \cap A \big| \bigg)
 \bigg] \le  \left( 1 +e^{kq_2(\gb)}\frac{ C k \log k \log M}{M}\right)^j\, ,
 \end{equation}
 for a $C>0$ depending only on $L(\cdot)$. It is in fact straightforward to 
 check that  with the choices made in \eqref{parameterz} the right-hand side is equal to $1+o(1)$ for $j=m$ (so  \eqref{eq:nts3}
 and therefore \eqref{eq:addGB2} hold true)
  if 
 \begin{equation}
\label{eq:condc2}
c_2\,>\, q_2(\gb)\,.
\end{equation} 
 
We prove \eqref{sqwash} by induction. Let us start with the simpler case $j=1$.
 We have 
 \begin{equation}
 \hat \bE_h \Big[
 \exp \Big( q_2(\gb) \big | \lint \tau_{1}+1, \tau_{2} \rint \cap A \big| \Big)
 \Big] \le 1+ e^{kq_2(\gb)} \hat \bP_h \big( \lint \tau_{1}+1,\tau_1+k \rint \cap A \ne \emptyset \big)\, .
 \end{equation}
Provided $k/M\le 1/2$ we have
\begin{equation}
\sum_{i=1}^{\infty} \sum_{n=a_i-k}^{b_i} K(n)\ind_{\{n\in [M,M^2]\}}
\le \sum_{i=1}^{M} \sum_{n=a_i-k}^{b_i} \frac{c}{n}\le  c \sum_{i=1}^{M} \sum_{n= iM-k}^{iM+k} \frac{1}{n}\le 
\frac{ 3c k \log M}{M}.
\end{equation}
 where in the first inequality  we used that $K(n)\le c n^{-1}$ and that $a_i\ge iM$  to show that all terms beyond the $M$ first ones do not participate in the sum.
 In the second inequality we used again the fact that $a_i\ge iM$ and $b_i-a_i\le k$ to bound each sum   $\sum_{n=a_i-k}^{b_i} (1/n)$. Now we observe that
 \begin{equation}
 \begin{split}
 \hat \bP_h \big( \lint \tau_{1}+1,\tau_1+k \rint \cap A \ne \emptyset \big)\, &=\, 
 \frac{
 \sum_{i=1}^{\infty} \sum_{n=a_i-k}^{b_i} K(n)\ind_{\{n\in [M,M^2]\}}
 }
 {\sum_{n=M}^{M^2}K(n)}\\ & \le\, \frac{ 4c\, k \log k \log M}{c_L \log 2\,  M}\,=: \frac{ C k \log k \log M}{ M}\, ,
  \end{split}
 \end{equation}
 where we have used \eqref{eq:asympt0.1} -- recall that $\log k = (1+o(1)) \vert \log h \vert$ -- and the last step defines
~$C$.

 \smallskip
 
 To prove the induction step, we condition with respect to $\tau_{2j}$.
 Let us consider 
 \begin{equation}
 A'\, :=\,  (A-\tau_{2j})\cap [M,\infty)\, .
 \end{equation} 
  Using independence of the jumps and the fact that $\tau_{2j+1}-\tau_{2j}\ge M$,
 \begin{multline}
  \hat \bE_h \bigg[
 \exp \bigg( q_2(\gb)\sum_{i=1} ^{j} \big|\lint \tau_{2j+1}+1, \tau_{2j+2} \rint \cap A \big| \bigg)\ \Big \vert \ \{\tau_{i}\}_{i\in \lint 0,2j\rint}
 \bigg] \\ =   \hat \bE_h \Big[
 \exp \Big( q_2(\gb) \big |\lint \tau_{1}+1, \tau_{2} \rint \cap A' \big| \Big) \Big].
 \end{multline}
Noting that $A'$ is of the type described in \eqref{leseta}-\eqref{spacingz} we can conclude that 
\begin{equation}
  \hat \bE_h \bigg[
 \exp \bigg( q_2(\gb)\sum_{i=1} ^{j} \big \vert \lint \tau_{2j+1}+1, \tau_{2j+2} \rint \cap A \big \vert   \bigg)\ 
 \Big\vert \ \{\tau_{i}\}_{i\in \lint 0,2j\rint}
 \bigg]\le  \left( 1 +e^{kq_2(\gb)}\frac{ C k \log k\log M}{M}\right),
\end{equation}
 which concludes the induction step.

 \subsubsection{Completion of the proof}
 Going back to \eqref{eq:PZ} and using \eqref{eq:addGB1} and \eqref{eq:addGB2}, which require 
 $c_1$ and $c_2$ to be  chosen according to \eqref{eq:condc1} and \eqref{eq:condc2},
 for $h $ sufficiently small
 we obtain
 \begin{align}
 \frac 1N \bbE \log Z_{N, \go}
 \, &\ge \, \frac m {N} \frac{(c_1- \upsilon-1 -\gep)}5 \log \log (1/h) - 2 \frac{\log N} N\,  \notag\\ 
 &\ge \frac 1{M^2 \log M}\, \ge\, h^2 
 \exp\left(-c_1 c_2 \frac {\log \log (1/h)}h\right)
 \,, 
 \end{align}
 and the proof of Theorem~\eqref{th:slowcase} is complete. 
 \qed

\section{Upper bound based on a global change of measure: proof of Theorem~\ref{th:general}}\label{genproof}

\label{sec:ub}
Instead of applying  Jensen's inequality directly   for $\bbE\log Z_{N,\go}$  (which yields the usual annealed bound), we can 
first split   $\log Z_{N,\go}$ in two terms and obtain 
\begin{equation}
\label{eq:fJ}
\bbE \log Z_{N,\go}=\bbE \log \left( f(\go) Z_{N,\go} \right)-  \bbE \log  f (\go) \, \le \, \log \bbE\left[ f(\go) Z_{N,\go}\right]\, - \bbE \log \left[ f(\go) \right] \, ,
\end{equation}
 where $f$ is an arbitrary positive measurable function.
We choose to apply the above inequality for a function $f$ which has the effect of penalizing  environments $\go$ which have small probability but contribute for most of the expectation  $\bbE[Z_{N,\go}]$: we select  a collection of events whose probability is small under $\bbP$, but large under the 
size-biased measure $\tilde \bbP_N(\dd \go):= \frac{Z_{N,\go}}{\bbE[Z_{N,\go}]} \, \bbP(\dd \go)$.  
What helps us in our choice of $f$ is that 
the size biased measure has a rather explicit description: it can be obtained starting with $\bbP$ and tilting the law of the $\go_n$  along a randomly chosen renewal trajectory.

\smallskip

We choose to penalize stretches of environment where $\go$ assumes  unusually large values. 
Let $k(h)$ be a large positive integer. The exact value of $k$ is to be fixed at the end of the proof as the result of an optimization procedure, 
but let us mention already that we choose it of the form $k(h)= h^{-1}\gp(h)$, where $\gp(h)\ge 1$ is a slowly varying function which tends to infinity when $h$ goes to zero.
For $n,m\in \lint 0,  N\rint$, with $m>n$,  and $b\in (0,1)$
 we introduce the  event
\begin{equation}
\label{eq:defce} 
\cE(n,m)= \cE_\gb(n,m) \, :=\, \bigg \{ \go:\, \sum_{j=n+1}^m
 \go_j \,\ge\, b \gl'(\gb) (m-n) \bigg \}\, .
\end{equation}
We define the \emph{penalizing} density
\begin{equation}
f(\go)\, :=\,
\exp \Bigg(- 4h %8 k^{-1}
 \sumtwo{n, m \in \lint 0, N\rint}{ n<m, \, m-n >k} (m-n)\ind_{\cE(n,m)} (\go) \Bigg)\, .
\end{equation}
To control the first term on the right in \eqref{eq:fJ}, it is sufficient to notice that by the standard Chernoff bound, there exists a constant
$c_1(\gb)$ (which can be taken as large as the Large Deviation function $\gS(b \gl'(\gb))$) such that for every $m\ge n$,
$\bbP\left(\cE(n,m)\right)\le e^{-c_1(\gb) (m-n)}.$
Hence 
\begin{align}
\bbE \log \left( f^{-1}(\go) \right)\, & =\, 
4h \sumtwo{n, m \in \lint 0, N\rint}{ n<m, \, m-n >k}  (m-n)
 \bbP\left(\cE (n,m)\right) 
 \le \, 4h  \sum_{n=0}^{N-1}\sum_{l>k}  e^{-c_1(\gb)l} \\
&\le \, \frac{4hN}{c_1(\gb)} %\frac{8N}{kc_1(\gb)}  
\exp(-c_1(\gb) k) \, \le \,  N \exp(-c_1(\gb) k)\,, 
\end{align}
where the last inequality requires $h\le c_1(\gb)/4$.
Therefore 
\begin{equation}
\label{eq:conclusion1}
\limsup_{N \to \infty}\frac 1N  \bbE\left[ \log f^{-1}(\go) \right]  \,\le \, \exp(-c_1(\gb) k)\, .
\end{equation}
\medskip

Now we turn to the second term in the right-hand side of \eqref{eq:fJ}, and prove that $\bbE\left[ f(\go) Z_{N,\go}\right]\le 1$ for $N$ large, provided that $h$ is sufficiently small and $k = h^{-1} \gp(h)$ with $\gp(\cdot)$ chosen in \eqref{def:phi}, so that the upper bound on the free energy is given by \eqref{eq:conclusion1}.
We have
\begin{equation}
\label{eq:second}
\bbE\left[ f(\go) Z_{N,\go}\right]\, = \, \bE \bigg[ \bbE \bigg[ 
f(\go) \exp \bigg( \sum_{n\in A} \left( \gb\go_n - \gl(\gb)\right)  \bigg) 
\bigg] \exp\left(h \vert A \vert \right); N \in \tau \bigg]\, ,
\end{equation}
where 
\begin{equation}\label{eq:defdea}
A= A(N)\, :=\, \bigcup_{j\in \lint 1,\cM_N\rint: \, \iota_j=1} \lint \tau_{j-1}+1,\tau_j\rint ,
\end{equation}
and $\cM_N$ is the integer such that $\tau_{\cM_N}=N$.
% a random set,  is the union of the  excursions below the interface. 
%  $A$ can be written as a finite  union of discrete intervals, each corresponding to an excursion: $A=\cup_j (l_j, r_j]\cap \bbN$,
% with $l_1<r_1<l_2<r_2< \ldots$. We now fix a realization of $A$, which of course corresponds to several 
% realization of $\tau$ 
%  and of the signs of the excursions. 
We have 
\begin{multline}
\label{eq:second2}
\bbE \bigg[ 
f(\go) \exp \bigg(\sum_{n\in A} \left( \gb\go_n - \gl(\gb)\right) \bigg)
\bigg]\\
 \le  \,
\bbE \bigg[ \bigg(
\prodtwo{j\in \lint 1,\cM_N\rint}{\iota_j=1, \tau_j-\tau_{j-1}\ge k} e^{-4h(\tau_j-\tau_{j-1}) \ind_{\cE (\tau_j,\tau_{j-1})}} \bigg)
\exp \bigg(\sum_{n\in A} \left( \gb\go_n - \gl(\gb) \right) \bigg)
\bigg] \\
= \,
\bbE \bigg[  \prodtwo{j\in \lint 1,\cM_N\rint}{\iota_j=1, \tau_j-\tau_{j-1}\ge k} 
e^{ -4h (\tau_j-\tau_{j-1}) \ind_{\cE (\tau_j,\tau_{j-1})}+\sum_{n\in \lint \tau_{j-1}+1,\tau_j \rint} \left( \gb\go_n - \gl(\gb)\right)}
\bigg]
\\  =  \,
\prodtwo{j\in \lint 1,\cM_N\rint}{\iota_j=1, \tau_j-\tau_{j-1}\ge k} \bbE_{\gb}\left[ 
e^{ -4h (\tau_j-\tau_{j-1}) \ind_{\cE (\tau_j,\tau_{j-1})}}
\right]
\, ,
\end{multline}
where in the last line the expectation is taken with respect to the probability measure $\bbP_{\gb}$ under which all the variables $\go_n$ are still IID, but with tilted
marginal density
\begin{equation}\label{eq:defpb}
\bbP_{\gb}(\go_1\in \dd u)\, =\, e^{\gb u-\gl(\gb)}\bbP(\go_1\in \dd u)\, .
\end{equation}

Note that the term $\exp\left(\sum_{n\in \lint \tau_{j-1}+1,\tau_j \rint} \left( \gb\go_n - \gl(\gb)\right)\right)$ is  a probability density.
In particular $\bbE_\gb[\go_n]=\gl'(\gb) >0$, and Chernoff bounds yields again for every $m\ge n$, 
\begin{equation}
\bbP_\gb\left( \cE (n,m)^{\cc}\right)\, \le\,  e^{-c_2(\gb)(m-n)}\, ,
\end{equation}
for an adequate choice of $c_2(\gb)>0$ that depends also on $b$.
% \begin{equation}
% \bbE\left[ 
% \exp\left( - 4h\ind_{\cE(0,d)} +\sum_{n\in (0, d]\cap \bbN} \left( \gb\go_n - \gl(\gb)\right) \right) 
% \right] \, =\, \bbE_{\gb}\left[ 
% \exp\left( -c_0 h d \ind_{\cE_\gep (0,d)}  \right)
% \right] \, ,
% \end{equation} 
% where, under $\bbP_\gb$, the $\go$ variables are still IID but their common law is tilted by the density 
% $\exp(  \gb\go - \gl(\gb))$.  and if $\gep =\gl' (\gb)/2$
% we readily see that $\bbP_\gb( \cE_\gep (0,d)) \ge 1- \exp(-c_2(\gb) d)$, for every $d$ larger than a constant 
% depending  only on $\gb$ and $\bbP$. 
Hence, for $h$ sufficiently small (depending on $c_2(\gb)$), for all $m-n>k$ we obtain 
\begin{equation}
\begin{split}
\bbE_{\gb}\left[ 
e^{- 4h(m-n) \ind_{\cE(n,m)}}\right] \, &=\, e^{- 4h(m-n)} +  \bbP_\gb \big( \cE (n,m)^{\cc} \big)(1-e^{- 4h(m-n)})
\\
& \le e^{- 4h(m-n)} + e^{-c_2(\gb)(m-n)}\, \le \, e^{-2h(m-n)}\, .
\end{split}
\end{equation}
Therefore, going back to \eqref{eq:second2}, we see that %(set $c_0=4$)
\begin{equation}
\bbE \bigg[ 
f(\go) \exp \bigg(\sum_{n\in A} \left( \gb\go_n - \gl(\gb)\right) \bigg) 
\bigg]
\, \le \, \prodtwo{j\in \lint 1,\cM_N\rint}{\iota_j=1, \tau_j-\tau_{j-1}\ge k}
\exp \big( -2 h (\tau_j-\tau_{j-1}) \big)\, ,
\end{equation}
and  this estimate, inserted  into  \eqref{eq:second},  is sufficient to transform the reward $(+h)$ into a penalty $(-h)$ for excursions of length larger than $k$.
We thus obtain
\begin{equation}
\label{eq:second3}
\begin{split}
\bbE\left[ f(\go) Z_{N,\go}\right]\,
&\le \, 
\sum_{n=1}^N \sum_{\ell \in \bbN ^n: \vert\ell \vert=N}
\prod_{j=1}^n K(\ell_j) \left( \frac 12 + \frac 12 e^{h \ell_j (\ind_{\{\ell_j\le k\}}-\ind_{\{\ell_j>k\}})}\right)
% \\
% &= \, 
% \sum_{n=1}^N \sum_{\ell \in \bbN ^n: \vert\ell \vert=N}
% \prod_{j=1}^n K(\ell_j) 
% \left( 1 - \frac 12 \left(1-e^{h \ell_j \ind_{\ell_j\le k}
% - h \ell_j \ind_{\ell_j>k}}\right)\right)
\\
&=: \, 
\sum_{n=1}^N \sum_{\ell \in \bbN ^n: \vert\ell \vert=N}
\prod_{j=1}^n K_k(\ell_j) \, ,
\end{split}
\end{equation}
To conclude we need to show that $K_k(\cdot)$ can be interpreted as the inter-arrival law for a renewal process -- this simply means that 
$\sum_{\ell=1}^\infty K_k(\ell) \le 1$ -- so that the last term in \eqref{eq:second3} is bounded by one because it 
is the probability that $N$ belongs to the  renewal with inter-arrival law $K_k(\cdot)$.
 We have therefore to establish the non positivity of  
\begin{equation}
\label{eq:sec4}
2 \Big(\sum_{\ell=1}^{\infty} K_k(\ell)-1 \Big) \, =\,\sum_{\ell \le k } K(\ell) \left( e^{h\ell} -1\right)-
 \sum_{\ell > k } K(\ell) \left( 1-e^{-h\ell} \right).
\end{equation}
To estimate the second term we recall that we have chosen $kh\ge 1$ so that when $k$ is sufficiently large we have 
\begin{equation}
  \sum_{\ell > k } K(\ell) \left( 1-e^{-h\ell} \right)\ge (1-e^{-1})\sum_{\ell>k} K(\ell)\ge \frac{1}{2}\tilde L(k).
\end{equation}
For the first term, using that $e^x-1\le e^X x$ for $x\in[0,X]$, we remark that
\begin{equation}
 \sum_{\ell \le k } K(\ell) \left( e^{h\ell} -1\right)\, \le\,  e^{hk}h \sum_{\ell\le k} L(\ell)\,\le\,  2e^{hk}hk L(k)\, ,
\end{equation}
where the last step holds (for $h$ sufficiently small) because
$\sum_{\ell\le k} L(\ell)= (1+o(1))k L(k)$ for $k \to \infty$, see \cite[Prop.~1.5.8]{cf:RegVar}.
Choosing $k = h^{-1} \gp(h)$ with $\gp(h) \ge 1$, we can conclude that the right hand side of \eqref{eq:sec4} is negative if 
\begin{equation}\label{linf}
 e^{\varphi(h)}\varphi(h)\frac{4 L(h^{-1}\varphi(h))}{ \tilde L(h^{-1}\varphi(h))}\le 1.
\end{equation}
Using the Potter bound \cite[Th.~1.5.6]{cf:RegVar}, we obtain that for $h$ sufficiently small
\begin{equation}\label{eq:sec6}
\frac{4 L(h^{-1}\varphi(h))}{ \tilde L(h^{-1}\varphi(h))}\le \varphi(h) \frac{L(h^{-1})}{\tilde L(h^{-1})},
\end{equation}
and thus \eqref{linf} is satisfied for all $h$ sufficiently small if one chooses (with $b$ as in \eqref{eq:defce})
\begin{equation}
\label{def:phi}
\gp(h)\, =\, b\,  \log\left( \frac{ \tilde L(1/h) }{ L(1/h)}\right) \, .
\end{equation}
 The estimate on the free energy is therefore determined by \eqref{eq:conclusion1}: 
the net result is  that  for $h>0$ sufficiently small 
\begin{equation}
\tf(\gb, h)\, \le \, \exp\left(- \gS\left(b\gl'(\gb)\right)\frac{\gp(h)}{h}\right)\, . 
\end{equation}
Since $b$ can be chosen arbitrarily close to one, 
the proof of Theorem~\ref{th:general} is complete.
\qed

\section{Improved upper bound: proof of the upper bounds in Theorem \ref{th:sharper}}
\label{sec:iub}

Choose $\gep \in (0,1)$ and, with reference to Definition~\ref{def:L},  define 
\begin{equation}
\label{eq:M}
M_h \, :=\, 
\begin{cases}
e^{ (1-\gep) v h^{-1} \log |\log h|} & \text{ sub-logarithmic case}\\
e^{(1-\gep)(\upsilon-1) h^{-1} |\log h| } & \text{ logarithmic case}\\
e^{(1-\gep) h^{-v/v-1}} & \text{ super-logarithmic case}
\end{cases}
\end{equation}

The heart of the proof is  the next proposition: we prove it after having shown that it implies the upper bounds we are after. Recall the definition \eqref{eq:q1q2} of $q_1(\gb)$.

\begin{proposition}
\label{prop:better}
Choose $L(\cdot)$ in the framework of Definition~\ref{def:L}, and $M_{\cdot}$ as in \eqref{eq:M}.
Then, for every $\gb\in (0,\bar \gb)$, and for every  $c_3<q_1(\gb)$, there exists $h_0>0$ such that for all $h\in (0,h_0)$
\begin{equation} 
\label{eq:better}
\bbE \big[ \log Z_{M_{h/c_3}, \go} \big] \, \le\,  4 \, .
\end{equation}
\end{proposition}

\noindent
\emph{Proof of Theorem \ref{th:sharper} (upper bounds).}
Proposition~\ref{prop:better} would directly imply the result if the sequence
$\{\bbE [\log Z_{N ,\go}]\}_{N=1,2, \ldots}$ were a sub-additive sequence. But this is not the case: rather, it is super-additive. However, given any $b>0$   and $L(\cdot)$,  one can choose two positive constants $c_3$ and $c_4$ 
such that the sequence formed by
\begin{equation}
\label{eq:GN}
\tg _N \, :=\, \bbE \left[\log Z_{N ,\go}\right] +c_4 \log N + c_5\, ,
\end{equation}
is sub-additive for every choice of $\gb \in [0, b]$ and $\vert h \vert\le b$. Therefore $\tf (\gb, h)=\lim_N \tg _N /N = \inf _N \tg _N /N\le \tg_M/M$ for any choice of $M$. 
Therefore, using \eqref{eq:better}, we have 
\begin{equation}
\label{eq:upperadditive}
\tf (\gb, h) \, \leq \,  \frac{4+c_4 \log M_{h/c_3} +c_5 }{M_{h/c_3}} \, ,
\end{equation}
which yields the  upper bounds in Theorem \ref{th:sharper}.

The proof that \eqref{eq:GN} forms a sub-additive sequence can be found for example in \cite[Ch.~4, \S\ 4.2]{cf:GB}, but we recall the argument for completeness. We start from the estimate that for every $N, M\in \bbN$
\begin{equation}
\label{eq:ubsupsub}
 \bbE \left[\log Z_{N+M ,\go}\right]\, \le \,  \bbE \left[\log Z_{N, \go}\right] +  \bbE \left[\log Z_{M, \go}\right]
 + C \log \left( \min(N, M)\right) +C\, ,
\end{equation}
with $C=C(b, L(\cdot))>0$ which can easily be made explicit
 %: the random variable $\go_1$ enters $C$ only via un upper bound on $\bbE \vert \go_1\vert$, but this upper bound is one by our choice $\bbE [\go_1^2]=1$.
A proof of \eqref{eq:ubsupsub} can be found for example in \cite[proof of (4.16), p.~93]{cf:GB}. Without loss of generality
we can assume $N\ge M$ and we rewrite \eqref{eq:ubsupsub}
as
\begin{equation}
\tg_{N+M}- \tg_N - \tg_M\, \le \,  C \log M +C + c_4 \log (N+M) -c_4 \log N -c_4 \log M -c_5\, ,
\end{equation}
and we are left with showing that the right-hand side is smaller or equal to zero, with suitable choices
of $c_4$ and $c_5$. For this we can  use $ \log (N+M)\le \log 2 + \log N$ and the fact that  the resulting expression 
vanishes if we choose $c_4=C$ and $c_5= (1+ \log 2)C$. 
\qed

\medskip

We now move to the proof of Proposition~\ref{prop:better} and 
as a first important step, we prove the following lemma.
\medskip

\begin{lemma}
\label{prop:smallj}
For any constant  $c_3<q_1(\gb)$, there exist $\eta \in(0,1)$ and  $h_1>0$ such that for all $h\in(0,h_1)$, setting $N= e^{c_{3}/h}$,  and $\theta=1- h/c_3$, we have that for every  $j\in \lint 1, N \rint$
\begin{equation}\label{lapremiere}
\bbE[Z_{j, \go}^{\theta}] \le e^{3} \, \widecheck \bP_h(j\in \tau) \, ,
\end{equation}
where $\widecheck \bP_h =  \widecheck \bP_h^{(\eta)}$ is the renewal process whose inter-arrival law is defined by
\begin{equation}
\label{eq:tau1tilde}
\widecheck \bP_h (\tau_1= \ell)=\widecheck K_h(\ell):= K(\ell)\Big( \frac12+  \frac12 e^{h \ell \big( \ind_{\{\ell\le  1/(\eta^2 h) \} }- \eta\, \ind_{\{\ell> 1/(\eta^2 h) \}} \big)}\Big) \, ,
\end{equation}
and satisfies (also for $h$ sufficiently small, how small depends on $\eta$)
\begin{equation}\label{lastitem}
 \sum_{\ell =1}^\infty \widecheck K_h(\ell) \le 1 - \frac16 \tilde L(1/h).
\end{equation}
\end{lemma}

\smallskip

 Note that this already proves that 
\begin{equation}\label{eq:somjen}
\bbE[\log Z_{N,\go}] \le \theta^{-1} \log \bbE[ Z_{N,\go}^{\theta}]\le  3\theta^{-1}\le 4
\end{equation} for $N= N_h:= e^{c_{3} /h}$, which in turn implies the upper bound, analogous to \eqref{eq:upperadditive}
\begin{equation}
\tf(\gb,h) \le \frac{4 +c_4 \log N_{h} +c_5 }{ N_h} \le e^{ - c'_{3} /h}\, .
\end{equation}
However, this bound alone is worse than the one achieved in Theorem \ref{th:general}.

\begin{proof}
The method we use to prove this statement presents some similarity with the proof of Theorem \ref{th:general}  (Section \ref{genproof}):
in particular it relies on the same notion of penalizing density. 
However, we need a different choice for $f(\go)$ in order to run computation of non-integer moments.

\smallskip

We fix $\eta>0$, we choose $k:= (\eta  h)^{-1}$ (so that $N=e^{-c_3 \eta k}$), and we work as if $1/\eta$ and $k$ were integers in the following. We define for $u \in \bbN$ (compare with  \eqref{eq:defce}) 
\begin{equation}
\cE_k(u) :=
\bigg \{  \sum_{i=k(u-1)+1}^{k u} \go_i \ge (1-\eta) \gl'(\gb) k  \bigg \}
\end{equation}
(we simply write  $\cE_k$ for $\cE_k(1)$),
and also
\begin{equation}
\label{def:f}
f(\go) = \prod_{u=1}^{ \lfloor N/k \rfloor } e^{- \frac{1+4\eta}{\eta}\, \ind_{\cE_k(u)}(\go)} \, .
\end{equation}

By using H\"older inequality, we get that for $\theta \in (0,1)$
\begin{equation}
\label{eq:Holder}
\bbE[(Z_{j,\go})^{\theta}] = \bbE \big[ f(\go)^{-\theta} (f(\go) Z_{j,\go})^{\theta} \big]  \le  
\bbE\big[ f(\go)^{-   \frac{\theta}{1-\theta}}\big]^{1-\theta} \bbE\big[ f(\go) Z_{j,\go}\big]^{\theta}\, .
\end{equation}

For the first term, a simple computation gives
\begin{align}
\bbE\big[ f(\go)^{-\theta/(1-\theta)}\big]^{1-\theta} = \bbE\Big[e^{ \frac{1+4\eta}{\eta} \, \frac{\theta}{1-\theta}\ind_{\cE_k}}  \Big]^{(1-\theta)\lfloor N/k \rfloor}
\le \Big( 1+ e^{\frac{1+4\eta}{\eta(1-\theta)}} \bbP\big( \cE_k\big)  \Big)^N \, .
\end{align}
From the standard Chernoff bound, for any $c_3<q_1(\gb)$, we can fix $\eta$ small, so that for all $k$  large enough, we have
$\bbP(\cE_k) \le e^{-  (1+5\eta) c_3 k }$. Here we use  $1-\theta  = h/c_3 $ and $k=(\eta h)^{-1}$, to obtain that 
\begin{equation}
 e^{\frac{1+4\eta}{\eta(1-\theta)}} \bbP\big( \cE_k\big) \,\le\,  e^{-   c_{3} \eta k} \, =\,  N^{-1}\, ,
 \end{equation}
and hence that 
\begin{equation}
\label{costchgmeas}
\bbE\big[ f(\go)^{-\theta/(1-\theta)}\big]^{1-\theta} \le \big(1+ N^{-1} \big)^N \le  e \, .
\end{equation}

To estimate the second term in \eqref{eq:Holder}, we observe that
\begin{equation}
\label{eq:EfZ}
\bbE\big[ f(\go) Z_{j,\go}\big] = \bE\bigg[ e^{h |A|} \ind_{\{j\in \tau\}} \bbE\Big[ f(\go) \exp\Big(\sum_{n\in A} \gb \go_n -\lambda(\gb) \Big) \Big] \bigg]\, ,
\end{equation}
where $A=A(j)$ is defined by \eqref{eq:defdea}.
If we set  
\begin{equation}
\mathcal{I}_A = \big\{ v \in \lint 1, \lfloor N/k \rfloor \rint \ : \  \lint k(v-1)+1,kv \rint \subset A \big\}\, ,
\end{equation}
by proceeding like  in \eqref{eq:second2} we obtain that
\begin{multline}
\bbE\Big[ f(\go) \exp\Big(\sum_{n\in A} \gb \go_n -\lambda(\gb) \Big) \Big] \le \bbE\Big[  \prod_{u\in \mathcal{I}_A} e^{-\frac{1+4\eta}{\eta}\, \ind_{\cE_k(u)}} 
e^{\sum_{n\in A}\gb \go_n -\lambda(\gb) }\Big] \\
\le \Big( e^{-\frac{1+4\eta}{\eta}}+ \bbP_{\gb}\big( (\cE_k)^{\cc} \big)  \Big)^{|\mathcal{I}_A|}\, ,
\end{multline}
where $\bbP_{\gb}$ is the tilted probability defined in \eqref{eq:defpb}.
As we have $\bbE_{\gb}[\go_n]=\gl'(\gb)$, by the law of large number, $\bbP_{\gb}((\cE_k)^{\cc})$ gets arbitrarily small if $k$ 
is chosen large (that is if $h$ is small). Hence for $h$ sufficiently small we obtain that 
\begin{equation}
\label{eq:addedbygb}
\bbE\Big[ f(\go) \exp\Big(\sum_{n\in A} \gb \go_n -\lambda(\gb) \Big) \Big]  \le e^{- \frac{1+3\eta}{\eta} |\mathcal{I}_A| } 
\, \le\, \exp\left(- \left(\frac{1+3\eta}{\eta}\right)  \left(\frac{(1/\eta) -1}{1/\eta } \right)\frac{ W_k}{k} \right) \, ,
\end{equation}
where (recall that $\cM_j$ is defined by $\tau_{\cM_j}=j$ if $j\in \tau$)
\begin{equation}
W_k \,=\,  W_k(\tau, \iota)\,:=\, \sum_{i\le \cM_j:\,  \iota_i=1} (\tau_i-\tau_{i-1})\ind_{\{\tau_i-\tau_{i-1}\ge k/\eta\}}\, ,
\end{equation}
is the sum of the lengths of the excursions below the interface that are of length $k/\eta$ or more.
Note that in \eqref{eq:addedbygb} we have used that an excursion of length $k/\eta$ or more covers at least $(1/\eta)-1$
consecutive $\lint1, k \rint$ blocks.
We therefore get from \eqref{eq:EfZ} and our choice $k= (\eta h)^{-1}$ that provided that $\eta$ is small,
\begin{multline} 
\label{fracmoment}
\bbE\big[ f(\go) Z_{j,\go}\big] \le \bE\Big[ e^{ h \big( |A| -  (1+\eta) W_k \big)} \ind_{\{j\in \tau\}} \Big] \\
= \sum_{q=1}^j \sumtwo{ (\ell_1, \ldots, \ell _q) \in \bbN^q:} {\sum_{i=1}^q \ell_i =N } \prod_{i=1}^q  K(\ell_i) \Big( \frac12 + \frac12  
e^{h \ell_i(\ind_{\{ \ell_i<k/\eta\}}  - \eta\, \ind_{\{\ell_i \ge k/\eta\}}) } \Big) = \widecheck \bP_h \big( j\in \tau \big) \, ,
\end{multline}
with $\widecheck \bP_h$ defined by \eqref{eq:tau1tilde}.
Going back to \eqref{eq:Holder}, and collecting \eqref{costchgmeas}-\eqref{fracmoment}, we end up with
\begin{equation}
\bbE[(Z_{j,\go})^{\theta}]  \,\le\, e\,  \widecheck \bP_h(j\in \tau)^{\theta} \, \le\, e\,  \widecheck \bP_h(j\in \tau)^{\theta-1} \widecheck \bP_h(j\in \tau)\, .
\end{equation}
Then, we simply notice that $\widecheck \bP_h (j\in\tau) \ge \frac12 	K(j) \ge N^{-2}$ for all $j \le N$ (provided that $N$ is large), so that 
$\widecheck \bP_h (j\in\tau)^{\theta-1} \le N^{ 2 (\log N)^{-1} } = e^{2}$.
This concludes the proof of~\eqref{lapremiere}.

% \begin{equation}
% \label{eq:tPh}
% \widecheck \bP_h \big(\tau_1=\ell \big) = K(\ell)  \Big( \frac12 + \frac12  e^{h \ell_i  - \frac16 \ell_i \ind_{\ell_i \ge 2k} } \Big)
% \end{equation}
% which is indeed a (sub-)probability, as shown in \eqref{eq:subproba} below.
% Now, we prove that for $h$ sufficiently small, we have
% \begin{equation}
% \label{eq:subproba}
% \sum_{\ell\ge 1} \widecheck \bP_h(\tau_1=\ell) \le 1 -  \frac16 \tilde L (1/h) \, .
% \end{equation}
It remains only to prove \eqref{lastitem}.
We have
\begin{equation}
\begin{split}
2 \left(\sum_{\ell = 1}^\infty \widecheck K_h(\ell)-1 \right) &=  \sum_{\ell \le  k/\eta} (e^{h \ell} -1 ) K(\ell) - \sum_{\ell > k/\eta} (1-e^{ - \eta h\ell}) K(\ell) \\
 &\le  e^{1/\eta^2} h \sum_{\ell \le k/\eta} \ell K(\ell) -  \frac{1}{2}\sum_{\ell > k /\eta} K(\ell)\\
 &\le 2 e^{1/\eta^2} h k L(k) - \frac2 5 \tilde L(k)  \le - \frac13 \tilde L(k)\, .
\end{split}
\end{equation}
For the first inequality, we used that that $e^{x}-1\le e^{1/\eta^2} x$ for all $0\le x \le 1/\eta^2$ (recall $hk=1/\eta$), 
and that $e^{ - \eta h\ell}\le e^{-1/\eta}\le 1/2$ in the second term.
The second inequality holds provided that $k$ is large enough, and the last one because $hk=1/\eta$ and $\tilde L(k) / L(k)$ diverges to infinity as $k\to\infty$. This completes the proof of Lemma~\ref{prop:smallj}.
\end{proof}

\begin{proof}[Proof of Proposition \ref{prop:better}]
Let $N= e^{ c_{3} / h}$ and $\theta = 1 - h/c_3 $, as in Lemma~\ref{prop:smallj}.

We prove that for all $1\le m\le M:=M_{h/c_3}$ we have $\bbE[ (Z_m)^{\theta}] \le e^{3}$, which gives the conclusion like in \eqref{eq:somjen}.
To that end as in \cite{cf:DGLT} we write for any $m\ge N$ 
\begin{equation}
Z_{m,\go} = \sum_{n=N}^m \sum_{j=1}^{N-1} Z_{m-n,\go} K(n-j) Z_{j,T^{m-j} \go},
\end{equation}
so that, using translation invariance, we have
\begin{equation}
\label{eq:decompZ}
\bbE\big[ (Z_{m,\go} \big)^\theta\big] \le  \sum_{n=N}^m \sum_{j=0}^{N-1}  \bbE\big[ ( Z_{m-n,\go})^{\theta} \big] K(n-j)^\theta \bbE\big[ ( Z_{j, \go} )^{\theta} \big]\, .
\end{equation}
We prove below that, for   $M= M_{h/c_{3}}$ with $M_h$ defined as in \eqref{eq:M}, we have
\begin{equation}
\label{key}
\rho:= \sum_{n=N}^{M} \sum_{j=0}^{N-1} K(n-j)^\theta \bbE\big[ ( Z_{j, \go} )^{\theta} \big] \le 1.
\end{equation}
Then, since from \eqref{eq:decompZ} we have
\begin{equation}
 \bbE\big[ (Z_{m,\go} \big)^\theta\big] \le \rho \max_{i\in\lint 0, m-N\rint} \bbE\big[ ( Z_{i,\go})^{\theta} \big]\, ,
\end{equation}
an easy induction gives that for any $N\le m\le M_{h/c_{3}}$ we have
\begin{equation}
\bbE[ (Z_{m,\go} \big)^\theta ] \le \max_{i\in \lint0, N\rint}  \bbE\big[ ( Z_{i, \go} )^{\theta} \big] \le e^{3} \, , 
\end{equation}
the last inequality coming from Lemma~\ref{prop:smallj}.

\smallskip
 To prove \eqref{key} we also use  Lemma~\ref{prop:smallj}: we assume that $N$ is even for simplicity  and
we write
\begin{equation}
\label{eq:A+B}
\frac{\rho}{e^{3}} \,  \le\,  
 \sum_{n=N}^{M}    \sum_{j=0}^{N/2-1} K(n-j)^\theta \widecheck \bP_h (j\in \tau)+
 \sum_{n=N}^{M}    \sum_{j=N/2}^{N-1} K(n-j)^\theta\widecheck \bP_h (j\in \tau)\, =:  A+B\, ,
\end{equation}
and  we estimate separately the two terms. We will show that, with our choice for $M_h$ in \eqref{eq:M}, both terms can be made arbitrarily small for $h\searrow 0$.

\smallskip
As far as $A$ is concerned, we observe that, since
we have $ \widecheck \bP_h (\tau_1=\infty)\ge \tilde L(1/h)/6$, as noted  in \eqref{lastitem}, we obtain
$\sum_{j=1}^\infty\widecheck \bP_h (j\in \tau)\le 7/ \tilde L(1/h)$ for $h$ sufficiently small. Hence, since $K(\cdot)$ is regularly varying with exponent $-1$, for $N$ large enough we have
\begin{equation}
\label{eq:Astep1}
A\, \le \, 3\, \frac{ 7}{ \tilde L(1/h)}  \sum_{n=N}^{M}  K(n)^\theta \, =\, 
\frac{21}{ \tilde L(1/h)}  \sum_{n=N}^{M} \frac{L(n)^{1-  h/c_3}}{n^{1- h/c_3}} \, .
\end{equation}
Using the fact that $L(n)/n$ is regularly varying (or the explicit value of $L$ in the framework of Definition~\ref{def:L}),
 we have for all $N$ and $M$ sufficiently large
\begin{equation}
\label{eq:Astep2}
 \sum_{n=N}^{M} \frac{L(n)^{1-h/c_3 }}{n^{1- h/c_3}} \, \le\, 2 \int_{N}^{M} 
 \frac{L(x)^{1- h/c_3}}{x^{1- h/c_3}}\dd x\, .
\end{equation}
Then, we define  $\psi(\cdot)$ by:
\begin{equation}
\psi(u)\, = \, (1-\gep) \times \begin{cases}
\upsilon \log \log u & \text{ in the \emph{sub-logarithmic} case,}
\\
(\upsilon-1) \log u & \text{ in the \emph{logarithmic} case,}
\\
u^{1/(\upsilon -1)} & \text{ in the \emph{super-logarithmic} case,}
\end{cases}
\end{equation}
so that  $M= \exp(c_3 \psi(c_3/h)/h)$.
Using the change of variable $x=e^{\frac{ c_{3} }{h} y}$, we have
\begin{equation}\label{bob}
  \int_{N}^{M} 
 \frac{L(x)^{1- h/c_3 }}{x^{1- h/c_3}}\dd x
 \, = \, 
 \frac{  c_3}{h}
 \int_{1}^{  \psi(h/c_3)}
  L\left( e^{\frac{ c_3}{h}y}\right)^{1- h/c_3} \exp(y)  \dd y\, .
\end{equation}
It is now a matter of direct evaluation of this last expression, with the help of Definition~\ref{def:L}.

\smallskip
(i) In the sub-logarithmic case, replacing $L$ in the integral by its asymptotic equivalent,
we obtain the following upper bound for the the r.h.s.\ of \eqref{bob}, valid for $h$ sufficiently small
\begin{equation}
\label{eq:Aslow}
\frac{2 c_L c_3}{h} \int_1^{(1-\gep)  \upsilon \log \log (c_3/h)}
\frac {h}{ c_3 y}
\big(  \log(c_3/h) \big)^{-\upsilon} e^{y } \dd y \, \le \,  4 c_L \frac{ \left(\log(c_3/h)\right)^{-\gep \upsilon }}{ \log \log(c_3/h)}  \, .
\end{equation}
We have used, for $a=1$, the inequality $\int_1^x e^z z^{-a}\dd z \le 2e^x x^{-a}$ valid for $x$ sufficiently large.
We conclude that $A$ is small by comparing the last term with $\tilde L(1/h)$ (cf.~\eqref{decaytilde}).

\smallskip
(ii) The same computation yields 
a similar upper-bound in the logarithmic case:
\begin{multline}
\frac{2c_L c_3}{h} \int_1^{(1-\gep)(\upsilon-1)  \log (c_3/h)}
\left( \frac{c_3 y}h\right)^{-\upsilon} e^y \dd y \\
\le 4
 c_L c_3^{1-\upsilon} h^{\upsilon-1} \left(\frac h {c_3}\right)^{-(1-\gep)(\upsilon-1)}
\frac{c_{\gep,\upsilon}}{ \log (c_3/h)}\le c'_{L,\gep}\frac{h^{\gep(\upsilon-1)}}{\log(1/h)}
\end{multline}
where $c_{\gep,\upsilon},c'_{L,\gep}>0$, and one can check again that the last term is much smaller than $\tilde L(1/h)$ for small values of $h$.

\smallskip
(iii) We are left with the super-logarithmic case. Using that $1- h/c_3 \ge 1/2$ for $h$ sufficiently small, we obtain the following upper bound
\begin{multline}
\label{eq:Afast}
\frac{2c_L c_3}{h} 
\int_1^{(1-\gep)   (c_3/h)^{1/(\upsilon -1)}}
\exp \Big(- 
 \frac12 \left(\frac{c_3 y}h\right)^{1/\upsilon} +y 
\Big) \dd y\\
\le \, \frac{2c_L c_3}{h}  \left(\frac{c_3}h\right)^{1/(\upsilon -1)}
\exp \Big(- 
c_{\upsilon} \left(\frac{c_3 }h\right)^{ 1/(\upsilon-1)} 
\Big)\, \le \, 
\exp \Big(- 
\frac{c_{\upsilon}}{2} \left(\frac{c_3 }h\right)^{1/(\upsilon-1)}
\Big)\, .
\end{multline}
In the first step we set $c_{\upsilon}:=(2\upsilon)^{\upsilon/(1-\upsilon)}(\upsilon -1)$, and used the fact that the argument of the exponential, for $y$ in the integration range, is maximal at $y=(2\upsilon)^{\upsilon/(1-\upsilon)} (c_3/h)^{1/(\upsilon -1)}$. The second inequality is valid for $h$ sufficiently small.
We conclude by observing that the last term is again of a smaller order than  $\tilde L(1/h)$ in that case.

\smallskip
Therefore, in view of \eqref{eq:Astep1}, \eqref{eq:Astep2} and \eqref{eq:Aslow}--\eqref{eq:Afast}, $A$ can made arbitrarily small, in particular smaller  smaller 
than  $e^{-3}/2$, in all cases.
Let us show that the same holds true also for $B$.
We will make use of the following Green function estimate.

\begin{lemma}
\label{th:trans-est}
There exists $C=C(\eta)>0$ such that 
\begin{equation}
\label{eq:trans-est}
\widecheck \bP_h ( n \in \tau) \, \le\,  C
\frac{K(n)}{\big(\tilde L(1/h)\big)^2}\, , 
\end{equation}
for every $n$ (in particular for $n\ge \exp(c_3/h)$).
\end{lemma}
\begin{proof}
The proof is derived from the following inequality: there is a constant $c>0$ (independent of $h$), such that for all $n\ge  k \ge 1$
\begin{equation}
\label{eq:general2}
\widecheck \bP_h ( \tau_k =n ) \le c k \widecheck \bP_h \big( \tau_1<\infty  \big)^k K(n)\, .
\end{equation}
Then, summing over $k$ gives the identity, since we have $\widecheck \bP_h \big( \tau_1<\infty  \big) \le 1- c / \tilde L(1/h)$, and  $\sum_{k=0}^{\infty}k (1-x)^k = 1/x^2$.

The proof of \eqref{eq:general2} follows from that of   \cite[Theorem 1.1-Equation (1.11)]{cf:AB} : one simply has to notice that Lemma 2.1 in \cite{cf:AB} 
is valid under the assumption that $\bP(\tau_1 =j )  \le c L(j)/j$ (the constant here depends on $\eta$ but not on $h$), and then all the  computations  
of Section~2.2 in \cite{cf:AB}   can be applied and yield \eqref{eq:general2}. This completes the proof of Lemma~\ref{th:trans-est}.
\end{proof}

Thanks to Lemma~\ref{th:trans-est}, and by using that $L(\cdot)$ is a slowly varying function, we obtain that for $N$ large enough we have $\widecheck \bP_h ( j \in \tau)\le 3C N^{-1}L(N)/ \tilde L(1/h)^2$ uniformly for $N/2\le j\le N-1$.
By considering the definition \eqref{eq:A+B} of $B$, we  obtain 
\begin{align}
B \, &\le 
\frac{3C L(N) }{N \tilde L(1/h)^2 }
 \sum_{n=N}^{M}  \sum_{j= N/2}^{N-1}   K(n-j)^\theta  \, .
\end{align}
Then we control
\begin{equation}
 \sum_{n=N}^{M}   \sum_{j= N/2}^{N-1}  K(n-j)^\theta \, \le\, 
  \sum_{t=1}^{N/2} t K(t)^{\theta} + \frac N2 \sum_{t=N/2}^{M}   K(t)^\theta  \le N \sum_{t=N/2}^{M}   K(t)^\theta \, ,
\end{equation}
where in the last inequality we used the fact that from slowly varying properties
\begin{equation}  
\sum_{t=1}^{N/2} tK(t)^{\theta}\le C N^{2-\theta} L(N)^{\theta} \, =\,  C e N L(N)^{\theta}\, ,
\end{equation}
which is negligible with respect to the second sum (we used the definition of $\theta$ to get that $N^{(1-\theta)} = e$).
We end up with
\begin{equation}  
B \le  
\frac{3C L(N) }{\tilde L(1/h) } \times \frac{1}{\tilde L(1/h) }
\sum_{n=N/2}^{M}    K(n)^\theta \, .
\end{equation}
Since we have already proven that the right-hand side in \eqref{eq:Astep1} vanishes
as $h$ becomes small, to prove that also $B$ vanishes in the same limit 
it suffices to show that 
$
{L(N)}/\tilde L(1/h)
$
is bounded for $h$ small. By recalling that $N=\exp(c_3/h)$, it is straightforward to see that in  the cases we consider, cf. Definition~\ref{def:L},  
such a ratio vanishes as $h$ goes to zero. Therefore also $B$ is under control and $\rho \le 1$ when $h$ is 
smaller than a well chosen constant. This completes the proof of Proposition~\ref{prop:better}.
\end{proof}

\section*{Acknowledgements}
This work has been performed in part when two of the authors, G.G. and H.L., were at the Institut Henri Poincar\'e (2017, spring-summer  trimester)
and we thank IHP for the hospitality. The visit to IHP by H.L. was supported by 
the Fondation de Sciences Math\'ematiques de Paris. G.G. acknowledges the support of grant  ANR-15-CE40-0020. H.L  acknowledges the support of a productivity grant from 
CNPq and a grant Jovem Cient\'sta do Nosso Estado from FAPERJ.

\end{document}